\documentclass{amsart}
\usepackage[margin=1.2in]{geometry}
\usepackage{amsmath,amsthm,amssymb,amsfonts}
\usepackage{xcolor}
\usepackage{mathrsfs}
\usepackage{verbatim}
\usepackage{float}
\usepackage{graphicx}
\usepackage{tikz-cd}
\usepackage{marvosym}
\usepackage{hyperref}
\usepackage[capitalize]{cleveref}
\usepackage{ytableau}
\usepackage{multicol}
\usepackage{dynkin-diagrams}
\usepackage[shortlabels]{enumitem}

\newcommand{\Z}{\mathbb{Z}}
\newcommand{\R}{\mathbb{R}}
\newcommand{\C}{\mathbb{C}}
\newcommand{\CP}{\mathbb{CP}}
\DeclareMathOperator{\Mod}{Mod}
\DeclareMathOperator{\SMod}{SMod}
\DeclareMathOperator{\Diff}{Diff}

\DeclareMathOperator{\Id}{Id}

\DeclareMathOperator{\sign}{sign}

\theoremstyle{definition}
\newtheorem{thm}{Theorem}[section]
\newtheorem{cor}[thm]{Corollary}
\newtheorem{prop}[thm]{Proposition}
\newtheorem{lem}[thm]{Lemma}

\newtheorem{example}[thm]{Example}
\newtheorem{question}[thm]{Question}

\theoremstyle{remark}
\newtheorem{rmk}[thm]{Remark}

\begin{document}
\title{Lefschetz fibrations with infinitely many sections}
\author{Seraphina Eun Bi Lee}
\address{Department of Mathematics, University of Chicago}
\email{seraphinalee@uchicago.edu}
\author{Carlos A. Serv\'an}
\address{Department of Mathematics, University of Chicago}
\email{cmarceloservan@uchicago.edu}
\maketitle
\vspace{-10mm}
\begin{abstract}
The \emph{Arakelov--Parshin} rigidity theorem implies that a holomorphic Lefschetz fibration $\pi: M \to S^2$ of genus $g \geq 2$ admits only finitely many \emph{holomorphic} sections $\sigma:S^2 \to M$. We show that an analogous finiteness theorem does not hold for smooth or for symplectic Lefschetz fibrations. We prove a general criterion for a symplectic Lefschetz fibration to admit infinitely many homologically distinct sections and give many examples satisfying such assumptions. Furthermore, we provide examples that show that finiteness is not necessarily recovered by considering a coarser count of sections up to the action of the (smooth) automorphism group of a Lefschetz fibration.
\end{abstract}

\section{Introduction}

Lefschetz pencils play a prominent role in the study of varieties, arising naturally in the projective setting by intersecting varieties with generic $1$-parameter families of hyperplanes \cite[Chapter 2]{Voisin-HT2}. Ever since Donaldson \cite{donaldson} showed that Lefschetz pencils
generalize to symplectic manifolds and Gompf \cite{gompf-symp} showed that any manifold admitting such a
structure is symplectic, Lefschetz pencils and fibrations have served as a
bridge between symplectic and K\"{a}hler 4-manifolds.
There exist symplectic Lefschetz fibrations that admit no complex structures; see \cite{ozbaci-stipsicz,korkmaz,baykur-noncomplex}. The study of differences between holomorphic and symplectic
Lefschetz fibrations
has been a fruitful area of research (e.g. ~\cite{stipsicz-chern-numbers,endo-nagami,baykur-korkmaz-simone,siebert-tian,stipsicz-indec}). In this paper we continue this study and examine the number of sections of symplectic and smooth $4$-dimensional Lefschetz fibrations. Our results contrast with the holomorphic setting.

\subsection{Counting sections of Lefschetz fibrations}
An important rigidity theorem in the
holomorphic case is the \emph{Arakelov--Parshin rigidity theorem} \cite{arakelov,parshin}, also called the
\emph{Geometric Shafarevich conjecture}, which states that there are only finitely many nontrivial families
of Riemann surfaces of genus $g \geq 2$ over a fixed Riemann surface of finite type.
The extra data of a \emph{section} of a Lefschetz fibration specifies a nontrivial family via the so called \emph{Parshin trick} (see \cref{intro:orbits}).
Thus Geometric Shafarevich implies the following finiteness theorem for holomorphic Lefschetz fibrations, which was originally proven independently by Manin and Grauert before the work of Arakelov and Parshin. See McMullen's survey \cite{mcmullen} for more details.
\begin{thm}[{(Special case of) Geometric Mordell Conjecture \cite{manin,grauert}}] A nontrivial, holomorphic Lefschetz fibration $M \to \CP^1$ of genus $g \geq 2$ admits finitely many holomorphic sections.
\end{thm}

Imayoshi--Shiga \cite{imayoshi-shiga} proved an extension of Geometric Shafarevich to families of finite type. A key fact used in their proof is that the monodromy representation of a nontrivial family is irreducible \cite[Section 3]{mcmullen}. Smith \cite[Proposition 4.2]{smith} later proved that in fact the same holds for the monodromy
representation of a smooth or symplectic Lefschetz fibration over $S^2$. Despite Smith's generalization, we show that Geometric Mordell does not hold for symplectic Lefschetz fibrations, i.e. smooth Lefschetz fibrations $\pi: M^4 \to S^2$ paired with symplectic structures $(M, \omega)$ for which (the smooth loci of) the fibers of $\pi$ are symplectic.
\begin{thm} \label{thm:cor-infty}
    For any $g \geq 2$, there exists a genus-$g$ Lefschetz fibration $\pi: M^4 \to S^2$ that admits infinitely many homologically distinct smooth sections. In other words, there exist sections $\sigma_k: S^2 \to M$ for $k \in \Z$ such that if $i \neq j$ then
  \[
    [\sigma_i(S^2)] \neq [\sigma_j(S^2)] \in H_2(M;\Z).
  \]
  Furthermore, $M$ admits a symplectic structure for which the fibers of $\pi$ and any section $\sigma_k(S^2)$ are all symplectic.
\end{thm}

We will deduce the existence of infinitely many symplectic sections from a more general criterion (Theorem \ref{thm:infinite-sections}) stated in \Cref{sec:lefschetz}.
In Section \ref{sec:construction} we describe a general procedure to smoothly construct these sections.
In \Cref{sec:symp}, we adapt the Gompf--Thurston construction to endow any Lefschetz fibration $\pi: M \to S^2$ with infinitely many sections $\sigma_k: S^2 \to M$ as constructed in \Cref{sec:construction}
with a symplectic form $\omega$ such that $\sigma_k(S^2)$ is a symplectic submanifold of $(M, \omega)$ for all $k\in \Z$.

\begin{rmk} It was claimed (\cite[Corollary 1.4, Theorem 5.1]{smith}) that Geometric Mordell also holds in the smooth category up to homotopy, i.e. that a Lefschetz fibration of genus $g \geq 2$ admits finitely many homotopy classes of smooth sections. There is a gap in the argument, as it claims (\cite[Line 12 of the proof of Theorem 5.1]{smith}) that the number of homotopy classes of sections is bounded above by the number of components of the complement of a fixed set of curves in minimal position in $\Sigma_{g,1}$ representing the vanishing cycles in a generic fiber. However, this presumes that \emph{any} lift of the monodromy representation from $\Mod(\Sigma_g)$ to $\Mod(\Sigma_{g,1})$ (where the marked point corresponds to the section in a regular fiber) is defined by Dehn twists about this fixed set of curves, which is not necessarily true (cf. Theorem \ref{thm:infinite-sections}\ref{thm:monodromy-factorization}).

On the other hand, Hayano \cite[Theorem 1.2]{hayano} showed that Geometric Mordell
does not hold in the more general setting of broken Lefschetz fibrations.
\end{rmk}

In Sections \ref{sec:apps} and \ref{sec:orbits} we apply Theorem \ref{thm:infinite-sections} to obtain many examples of Lefschetz fibrations admitting infinitely many homologically distinct sections. For example, the following corollary supplies many decomposable examples of such fibrations.
\begin{cor}[Decomposable examples]\label{cor:intro-untwisted}
  Let $\pi: M \to S^2$ be a nontrivial Lefschetz fibration of genus $g \geq 2$ that admits a section. The untwisted fiber sum $\pi \#_F \pi: M \#_FM \to S^2$ admits infinitely many homologically distinct smooth sections.
  Furthermore, $M\#_F M$ admits a symplectic structure for which the fibers of $\pi$ and infinitely many sections are all symplectic.
\end{cor}
\Cref{cor:trivial-homology} shows that the same holds for fiber sums $M_1 \#_{F,\psi}M_2$ if each factor has trivial first homology. On the other hand, the following corollary shows that our results are not restricted to decomposable fibrations.
\begin{cor}[Indecomposable examples]\label{cor:intro-indec}
For every $g \geq 2$, there exists a genus-$g$, fiber sum indecomposable, symplectic Lefschetz fibration that admits infinitely many homologically distinct symplectic sections.
\end{cor}

There also exist many holomorphic Lefschetz fibrations to which Theorem \ref{thm:infinite-sections} applies. For example, if $\pi: M \to S^2$ is a holomorphic Lefschetz fibration of genus $g \geq 2$ then the untwisted fiber sum $\pi \#_F \pi: M \#_FM \to S^2$ also admits a holomorphic structure. Thus, coupling \Cref{cor:intro-untwisted} with Geometric Mordell gives the following corollary.
\begin{cor}
  Let $\pi: M \to \CP^1$ be any nontrivial, holomorphic Lefschetz fibration of genus $g \geq 2$ that admits a section. For any holomorphic structure of the untwisted fiber sum $\pi \#_F \pi: M \#_F M \to \CP^1$, there exist infinitely many homologically distinct smooth sections that are not homologous to any holomorphic section.
\end{cor}

The construction of the sections given by Theorem \ref{thm:infinite-sections} requires the existence of at least one section. It is unknown whether every Lefschetz fibration admits a section. Our work thus prompts the following question.

\begin{question}
  Are there Lefschetz fibrations $\pi: M \to S^2$ with a finite, positive number of homotopically distinct smooth sections? If so, is there an effective bound on the number of sections?
\end{question}

Finally, we note that the existence of infinitely many homologically distinct symplectic sections generalizes to certain Lefschetz fibrations $\pi: M \to B$ over arbitrary surfaces $B$, including some surface bundles over surfaces, because the construction of the sections in Section \ref{sec:construction} is local. See Remark \ref{rmk:surface-bundle} for more details.

\subsection{Isomorphism classes of sections and the Parshin trick}\label{intro:orbits}

The \emph{Parshin trick} is another input to the proof of Geometric Mordell in the holomorphic category. It relates a holomorphic section $s:B\to C$ of a family $C \to B$ with fibers of genus $g$ to a family $D \to B$ with fibers of higher genus. Parshin's (and Kodaira's \cite{kodaira}) idea is
to use the section to create a branched cover $D \to C$ branched over $s(B)$ and form a family $D \to C \to  B$. For more details, see \cite[Section 4]{mcmullen}.

\begin{thm}[{Parshin trick \cite{parshin}}]
  Given a genus $g \geq 1$ and a base $B$ of finite type, there exists a genus $h \geq 2$ and a finite-to-one map
  \[ \left\{ \parbox{18em}{\centering Families $C \to B$ with fibers of genus $g$, equipped with sections $s:B \to C$} \right\} \biggr / \sim \,\,  \to \left\{ \parbox{10em}{\centering Families $D \to B$ with fibers of genus $h$}\right\} \biggr / \sim.
  \]
  Here, two families $\pi_1: C_1 \to B$ and $\pi_2: C_2 \to B$ are \emph{equivalent} if there is an isomorphism $f: C_1 \to C_2$ lying over the identity of $B$, i.e. $\pi_1 = \pi_2 \circ f$. Two families with sections $(C_1, s_1)$ and $(C_2, s_2)$ are equivalent if there is an isomorphism $C_1 \cong C_2$ of families  over $B$ sending $s_1$ to $s_2$.
\end{thm}
The Parshin trick, in conjunction with Geometric Shafarevich, implies Geometric Mordell in the case of fiber genus $g \geq 2$ because the automorphism group of a family with fiber genus $g \geq 2$ is finite. In light of the Parshin trick, we incorporate the action of the smooth automorphism group of a Lefschetz fibration $\pi: M \to S^2$ in our count of its sections. For two sections $s_1, s_2$ of $\pi$, the pairs $(\pi, s_1)$ and $(\pi, s_2)$ are \emph{isomorphic} if there exist orientation-preserving diffeomorphisms $\Psi: M \to M$ and $\psi: S^2 \to S^2$ such that the following diagram commutes:
\begin{center}
\begin{tikzcd}
M \arrow[r, "\Psi"] \arrow[d, "\pi"]              & M \arrow[d, "\pi"']               \\
S^2 \arrow[r, "\psi"] \arrow[u, "s_1", bend left] & S^2 \arrow[u, "s_2"', bend right]
\end{tikzcd}
\end{center}
\noindent
The pair $(\Psi, \psi)$ is an \emph{automorphism} of the Lefschetz fibration $\pi: M \to S^2$.

In contrast to the holomorphic setting, we show through explicit examples that there is no finiteness result even for the number of sections up to isomorphism (not necessarily lying over the identity) in the smooth setting.
\begin{thm}\label{thm:intro-non-isomorphic}
For any $g \geq 2$, there exists a genus-$g$ Lefschetz fibration $\pi: M \to S^2$ with infinitely many homologically distinct smooth sections that are pairwise non-isomorphic.
\end{thm}

We also show for an abundant family of examples that there is no finiteness result for the number of sections up to isomorphisms covering the identity in the smooth category, in direct contrast with the Parshin trick. (See Proposition \ref{prop:covering-identity}.) The following questions arise naturally.

\begin{question}
  Are there Lefschetz fibrations with finitely many isomorphism classes of smooth sections despite admitting infinitely many homotopically distinct sections? Are there any such Lefschetz fibrations with a unique isomorphism class of smooth sections?
\end{question}

We give examples in which the sections arising from our construction lie in a \emph{single} isomorphism class of sections (Example \ref{ex:surj-monodromy}).
Yet this does not address the last question, as we cannot rule out the presence of sections unrelated to our construction.

\subsection{Organization of the paper}
In \Cref{sec:lefschetz} we recall basic facts about Lefschetz fibrations, set up notation, and state our most general theorem (\cref{thm:infinite-sections}).
In \Cref{sec:construction} we describe our construction of sections $\sigma_k:S^2 \to M$, to be used in the proof of \Cref{thm:infinite-sections}. In \Cref{sec:proof-main-thm} we show that the sections $\sigma_k$ are homologically distinct. In \Cref{sec:symp}, we extend our results to the symplectic category by a slight modification of the Gompf--Thurston construction and conclude the proof of \Cref{thm:infinite-sections}. \Cref{sec:apps} contains many examples of Lefschetz fibrations (both decomposable and indecomposable) satisfying the assumptions of \Cref{thm:infinite-sections} and proves \Cref{thm:cor-infty} and Corollaries \ref{cor:intro-untwisted} and \ref{cor:intro-indec}.
Finally, in \Cref{sec:orbits} we study the action of the automorphism group of $\pi:M \to S^2$
on our construction and prove \Cref{thm:intro-non-isomorphic}.

\subsection{Acknowledgments.}
We are grateful to our advisor Benson Farb for his invaluable guidance and support throughout this project and for his numerous comments on earlier drafts of this paper. We would like to thank Ivan Smith for helpful email correspondences regarding his work \cite{smith}, \.{I}nanç Baykur for many insightful suggestions, including some that led us to the proof of Corollary \ref{cor:intro-indec}, and Linus Setiabrata for suggesting the generalization to Lefschetz fibrations over arbitrary bases (Remark \ref{rmk:surface-bundle}). We also thank Nick Salter for a helpful email correspondence and comments on an earlier draft and Bena Tshishiku for an insightful conversation.

\section{Lefschetz fibrations and sections}\label{sec:lefschetz}

The goal of this section is to state our main tool (Theorem \ref{thm:infinite-sections}) to detect the existence of infinitely many sections of certain genus-$g$ Lefschetz fibrations $\pi: M \to S^2$. Before doing so, we fix notation and state our standing assumptions.

\subsection{Lefschetz fibrations and monodromy}
Let $M^4$ be a closed, connected, oriented, smooth $4$-manifold. A surjective smooth map $\pi: M \to S^2$ is a \emph{Lefschetz fibration} if $\pi$ has finitely many critical points $q_1, \dots, q_r \in M$ and for each critical point $q_i$, there are smooth, orientation-compatible charts $U_i \cong \C^2$ around $q_i \in M$ and $V_i \cong \C$ around $\pi(q_i) \in S^2$ such that relative to these charts, $\pi$ takes the form
\[
	\pi(z_1, z_2) = z_1^2 + z_2^2.
\]
A Lefschetz fibration is \emph{nontrivial} if it has a positive number of critical points. We assume that Lefschetz fibrations are injective on the set of their critical points and that Lefschetz fibrations are \emph{relatively minimal}, meaning that no fiber of $\pi$ contains an embedded $(-1)$-sphere. The \emph{genus} of a Lefschetz fibration $\pi: M \to S^2$ is the genus of the surface $\pi^{-1}(b)$ for any regular value $b \in S^2$. Note that the orientations of $M$ and $S^2$ determine an orientation on any regular fiber $\pi^{-1}(b)$. In particular, if $s: S^2 \to M$ is a section of $\pi$ then
\[
	Q_{M}([\pi^{-1}(b)], [s(S^2)]) = 1
\]
where $Q_{M}: H_2(M; \Z) \times H_2(M; \Z) \to \Z$ denotes the algebraic intersection form of $M$.

Let $\pi_1: M_1 \to S^2$ and $\pi_2: M_2 \to S^2$ be genus-$g$ Lefschetz fibrations admitting sections $s_1: S^2 \to M_1$ and $s_2: S^2 \to M_2$ respectively. For each $i = 1, 2$, let $\nu_i \subseteq M_i$ be a fiberwise tubular neighborhood of a regular fiber of $\pi_i$. For any fiberwise orientation-reversing diffeomorphism of pairs $\psi: (\partial \nu_1, s_1(\partial \nu_1)) \to (\partial \nu_2, s_2(\partial \nu_2))$, consider the \emph{fiber sum}
\[
	\pi_1 \#_{F, \psi} \pi_2: M_1 \#_{F, \psi} M_2 \to S^2
\]
which is a genus-$g$ Lefschetz fibration obtained by gluing $M_1 - \nu_1$ and $M_2 - \nu_2$ along their boundaries via $\psi$. The fiber sum $\pi_1 \#_{F, \psi} \pi_2$ naturally also admits a section $s_1 \#_{F, \psi} s_2: S^2 \to M_1 \#_{F, \psi} M_2$ defined piecewise as $s_i$ on the image of $M_i - \nu_i$ in $S^2$ for each $i = 1, 2$. If $(M_1, \nu_1) = (M_2, \nu_2)$ and $\psi$ restricts to the identity map on some fiber $\pi_1^{-1}(q) = \pi_2^{-1}(q)$ then we omit $\psi$ from the notation and $\pi_1 \#_F \pi_2: M_1 \#_F M_2 \to S^2$ is called an \emph{untwisted} fiber sum.

The data $(\pi: M \to S^2, s: S^2 \to M)$ of a Lefschetz fibration and section can be encoded in an antihomomorphism called the \emph{monodromy representation} \cite[p.~291]{gompf-stipsicz}. The choices of a regular value $b \in S^2$ and a diffeomorphism of pairs $\Phi_b: (\pi^{-1}(b), s(b)) \xrightarrow{\sim} (\Sigma_g, p)$ for a marked point $p \in \Sigma_g$ determine a monodromy representation of $(\pi, s)$
\[
	\rho_{(\pi, s)}: \pi_1(S^2 - \{q_1, \dots, q_r\}, b) \to \Mod(\Sigma_{g,1}),
\]
where $q_1, \dots, q_r \in S^2$ denote the singular values of $\pi$ and $\Sigma_{g,1}$ is a genus-$g$ surface with one marked point $p$. If $\gamma_i \in \pi_1(S^2 - \{q_1, \dots, q_r\}, b)$ is a loop obtained from a small, counterclockwise loop around $q_i$ connected to $b$ by a path in $S^2 - \{q_1, \dots, q_r\}$, the monodromy $\rho_{(\pi, s)}(\gamma_i)$ is a right-handed Dehn twist $T_{\ell_i} \in \Mod(\Sigma_{g,1})$ about a \emph{vanishing cycle} $\ell_i$, which is an isotopy class $\ell_i$ of some essential simple closed curve in $\Sigma_{g,1}$. Fixing a choice of generators $\gamma_1, \dots, \gamma_r \in \pi_1(S^2 - \{q_1, \dots, q_r\}, b)$ whose composition gives a contractible loop
encircling all singular values determines a \emph{monodromy factorization} of the pair $(\pi, s)$, which is a relation in $\Mod(\Sigma_{g,1})$ of the form
\[
	T_{\ell_r} \dots T_{\ell_1} = 1 \in \Mod(\Sigma_{g, 1}).
\]

The group $\Mod(\Sigma_{g,1})$ fits into the \emph{Birman exact sequence} \cite[Section 4.2]{farb-margalit}
\begin{equation}\label{eqn:birman}
	1 \to \pi_1(\Sigma_g, p) \xrightarrow{\mathrm{Push}} \Mod(\Sigma_{g,1}) \xrightarrow{\mathrm{Forget}} \Mod(\Sigma_g) \to 1.
\end{equation}
For any $\gamma \in \pi_1(\Sigma_g, p)$, denote
\[
	P_\gamma := \mathrm{Push}(\gamma^{-1}) \in \Mod(\Sigma_{g,1})
\]
so that $P_\gamma$ is represented by a diffeomorphism of $\Sigma_{g,1}$ obtained by ``pushing $p$ along $\gamma$.'' We note that the map $\gamma \mapsto P_\gamma$ is an antihomomorphism of groups.
\subsection{Statement of the main theorem}
We are now ready to state a sufficient condition for a symplectic Lefschetz fibration to admit infinitely many homologically distinct symplectic sections, which can be checked directly from a monodromy factorization of a section.
\begin{thm}\label{thm:infinite-sections}
Let $\pi: M \to S^2$ be a genus-$g$ Lefschetz fibration with $g \geq 2$ and let $s: S^2 \to M$ be a section of $\pi$. Fix a subword $(T_{\ell_{r_1}} \dots T_{\ell_1})$ of a monodromy factorization
\[
	T_{\ell_{r_1+r_2}} \dots T_{\ell_{r_1+1}} T_{\ell_{r_1}} \dots T_{\ell_1} = 1 \in \Mod(\Sigma_{g,1})
\]
of the pair $(\pi, s)$. Suppose that the following conditions hold:
\begin{itemize}
\item There exists a nonseparating simple closed curve $\delta \subseteq \Sigma_{g, 1}$ such that $[\delta] \in H_1(\Sigma_g; \Z)$ is an element of
\[
	\Z\{[\ell_1], \dots, [\ell_{r_1}]\} \cap \Z\{[\ell_{r_1+1}], \dots, [\ell_{r_1+r_2}]\} \subseteq H_1(\Sigma_g; \Z).
\]
\item There exists a simple loop $\gamma: \R/\Z \to \Sigma_g$ with $\gamma(0) = p \in \Sigma_g$ that intersects $\delta$ transversely exactly once in $\Sigma_g$ and
\[
	(T_{\ell_{r_1}} \dots T_{\ell_{1}})([\gamma]) = [\gamma] \in \pi_1(\Sigma_g, p).
\]
\end{itemize}
Then the Lefschetz fibration $\pi: M \to S^2$ admits infinitely many sections
\[
	\{\sigma_k : S^2 \to M : k \in \Z\}
\]
such that
\begin{enumerate}[(a)]
\item the sections $\sigma_k$ are pairwise homologically distinct, i.e. if $k_1 \neq k_2 \in \Z$ then
\[
	[\sigma_{k_1}(S^2)]\neq [\sigma_{k_2}(S^2)] \in H_2(M; \Z);\label{thm:homology-distinct}
\]
\item a monodromy factorization of $(\pi, \sigma_k)$ is
\[
	\left(T_{\ell_{r_1+r_2}} \dots T_{\ell_{r_1 + 1}}\right) \left(T_{P_\gamma^k(\ell_{r_1})}\dots T_{P_\gamma^k(\ell_1)}\right)= 1 \in \Mod(\Sigma_{g,1});\label{thm:monodromy-factorization}
\]
\item there exists a symplectic form $\omega$ of $M$ such that (the smooth loci of) the fibers of $\pi$ and any section $\sigma_k(S^2)$ are all symplectic submanifolds of $(M, \omega)$. \label{thm:symplectic-form}
\end{enumerate}
\end{thm}

\begin{rmk}
  The loop $\gamma$ will be used to construct the sections $\sigma_k$, and
  because of smoothness we will assume that $\gamma(t) = p$ for $t$ near $0 \in \R/\Z$. The curve $\delta \subseteq \Sigma_{g, 1}$ will be used to detect that the sections $\sigma_k$ are pairwise homologically distinct. Also see Remark \ref{rmk:delta-geom} for a topological interpretation of the assumption on $[\delta]$.
\end{rmk}

The proofs of Theorem \ref{thm:infinite-sections}\ref{thm:homology-distinct}, \ref{thm:monodromy-factorization}, and \ref{thm:symplectic-form} can be found in Sections \ref{sec:main-proof}, \ref{sec:monodromy-factorization}, and \ref{sec:symp} respectively. Each part is proven independently from the rest.

\section{Construction of sections $\sigma_k$ for each $k \in \Z$}\label{sec:construction}

In this section we construct a section $\sigma_k: S^2 \to M$ of $\pi: M \to S^2$ for each integer $k \in \Z$, to be used in the proof of Theorem \ref{thm:infinite-sections}, emphasizing a special case in which $\pi$ is a fiber sum $\pi_1 \#_{F,\psi} \pi_2: M_1 \#_{F, \psi} M_2 \to S^2$ which will be used extensively in later sections.
\subsection{A decomposition of $S^2$}
Fix the notation and assumptions of Theorem \ref{thm:infinite-sections}. Let $q_1, \dots, q_{r_1+r_2} \in S^2$ denote the singular values of $\pi$ corresponding to the vanishing cycles $\ell_1, \dots, \ell_{r_1+r_2}$ respectively.

The base $S^2$ can be written as a union of three sets
\[
	S^2 = D_1 \cup A \cup D_2,
\]
such that (cf. Figure \ref{fig:s2-decomposition})
\begin{enumerate}[(a)]
\item $D_1, D_2 \cong D^2$ are closed $2$-disks and $A = S^1 \times (0, 1)$ is an open annulus so that $\partial D_1 = S^1 \times \{0\}$ and $\partial D_2 = S^1 \times \{1\}$ in the closure $\overline A = S^1 \times [0, 1]$,
\item the base point $b$ is contained in $\partial D_2$,
\item $q_1, \dots, q_{r_1}$ are contained in $D_1$ and the singular values $q_{r_1+1}, \dots, q_{r_1+r_2}$ is contained in $D_2$, and
\item the based loop $(C, b) := (\partial D_2, b) \subseteq (S^2, b)$ can be oriented so that when considered as an element of $\pi_1(S^2 - \{q_1, \dots, q_{r_1+r_2}\}, b)$,
\[
	\rho_{(\pi,s)}(C) = T_{\ell_{r_1}} \dots T_{\ell_1} \in \Mod(\Sigma_{g, 1})
\]
where $\rho_{(\pi,s)}: \pi_1(S^2 - \{q_1, \dots, q_{r_1+r_2}\}, b) \to \Mod(\Sigma_{g,1})$ is the monodromy representation of $(\pi, s)$ with respect to a diffeomorphism of pairs $\Phi_b: (\pi^{-1}(b), s(b)) \to (\Sigma_g, p)$.
\end{enumerate}
\begin{figure}
\centering
\includegraphics[width=0.25\textwidth]{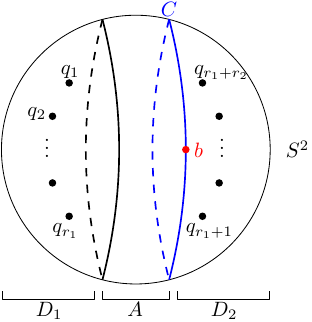}
\caption{This figure summarizes the notation for various parts of $S^2$.}\label{fig:s2-decomposition}
\end{figure}
Let
\[
	X_1 := \pi^{-1}(D_1\cup \overline A), \qquad X_2 := \pi^{-1}(D_2).
\]
\begin{rmk}\label{rmk:delta-geom}
In the statement of Theorem \ref{thm:infinite-sections}, $[\delta] \in H_1(\Sigma_g; \Z)$ is assumed to be contained in the intersection $\Z\{[\ell_1], \dots, [\ell_{r_1}]\} \cap \Z\{[\ell_{r_1+1}], \dots, [\ell_{r_1+r_2}]\}$. This is is equivalent to the condition that $[\delta] = 0\in H_1(X_i; \Z)$ for both $i = 1, 2$, which can be seen by viewing $X_i$ as $\Sigma_g \times D^2$ with a $2$-handle attached along each vanishing cycle $\ell_k$ of $X_i$ (\cite[Section 8.2]{gompf-stipsicz}) and applying Mayer--Vietoris.
\end{rmk}

\begin{example}\label{ex:untwisted}
  Let $\pi_1: M_1 \to S^2$ and $\pi_2: M_2 \to S^2$ be genus-$g$ Lefschetz fibrations with sections $s_1: S^2 \to M_1$ and $s_2: S^2 \to M_2$. Consider any fiber sum $\pi_1 \#_{F,\psi} \pi_2: M_1 \#_{F,\psi} M_2 \to S^2$ and the corresponding section $s_1 \#_{F,\psi} s_2$; denote this pair by $(\pi, s)$ and its monodromy representation by $\rho_{(\pi,s)}$. Let
\[
	\left(T_{\ell_{r_1+r_2}} \dots T_{\ell_{r_1+1}}\right) \left(T_{\ell_{r_1}} \dots T_{\ell_1}\right)  = 1\in \Mod(\Sigma_{g,1})
\]
denote a monodromy factorization of $(\pi, s)$ such that
\[
	T_{\ell_{r_1}} \dots T_{\ell_1} = T_{\ell_{r_1+r_2}} \dots T_{\ell_{r_1+1}} = 1 \in \Mod(\Sigma_{g,1})
\]
are factorizations of $(\pi_1,s_1)$ and $(\pi_2, s_2)$ respectively.

For the subword $T_{\ell_{r_1}} \dots T_{\ell_1}$, arrange so that
\[
	X_1 = M_1 - \nu_1, \qquad X_2 = M_2 - \nu_2
\]
where $\nu_i \subseteq M_i$ is a fiberwise neighborhood of a regular fiber of $\pi_i$ for both $i = 1, 2$ so that $M_1 \#_{F,\psi} M_2 = X_1 \cup_\psi X_2$.
In particular, $\{\ell_1,\ldots, \ell_{r_1}\}$ is a set of vanishing cycles of $M_1$ and $\{\ell_{r_1+1}, \ldots, \ell_{r_1 + r_2}\}$ is a set of vanishing cycles of $M_2$.

It follows that for $C$ as above, $\rho_{(\pi,s)}(C) = 1 \in \Mod(\Sigma_{g,1})$. Thus, any $\gamma \in \pi_1(\Sigma_g, p)$ is fixed by $\rho_{(\pi, s)}(C) = T_{\ell_{r_1}} \dots T_{\ell_1}$.
\end{example}

\subsection{Construction of the sections $\sigma_k$}\label{sec:sections}

In this section we construct the sections $\sigma_k: S^2 \to M$ of $\pi: M \to S^2$. The following example first illustrates a special case in which $\pi = \pi_1 \#_{F,\psi} \pi_2: M_1 \#_{F,\psi} M_2 \to S^2$ is a fiber sum, extending the setup described in Example \ref{ex:untwisted}.

\begin{example}\label{ex:sections-fiber}
Consider the fiber sum $\pi = \pi_1 \#_{F,\psi} \pi_2: M_1 \#_{F,\psi} M_2 \to S^2$ as in Example \ref{ex:untwisted}. Because $\rho_{(\pi, s)}(C) = 1$, there is a diffeomorphism $\Phi: \pi^{-1}(\overline A)\to (\Sigma_g \times S^1) \times [0, 1]$ such that the following diagram commutes:
\begin{figure}[H]
\centering
\begin{tikzcd}
\pi^{-1}(\overline A) \arrow[d, "\pi"] \arrow[rr, "\Phi"] &  & \Sigma_g \times S^1 \times [0, 1] \arrow[d, "{(x, \theta, t) \mapsto (\theta, t)}"] \\
\overline A \arrow[rr, Rightarrow, no head] \arrow[u, "s", bend left]       &  & S^1 \times [0, 1] \arrow[u, "{(\theta, t) \mapsto (p, \theta, t)}", bend left]
\end{tikzcd}
\end{figure}
\noindent
Moreover, we may arrange so that $\Phi$ extends $\Phi_b$, i.e. $\Phi|_{\pi^{-1}(b)} = \Phi_b \times \{b\}$ where $b \in \overline A = S^1 \times [0, 1]$ is the fixed regular value of $\pi$.

For all $k \in \Z$, let $\sigma_k: S^2 \to M_1 \#_{F,\psi} M_2$ be defined by
\[
	\sigma_k(x) = \begin{cases}
	s(x) & \text{if }x \in D_1 \cup D_2, \\
	\Phi^{-1}(\gamma(tk), \theta, t) & \text{if }x = (\theta, t) \in \overline A.
	\end{cases}
\]
See Figure \ref{fig:section}. Because $(\gamma(tk), \theta, t) = (p, \theta, t)$ for any $(\theta, t) \in \partial A$, the map $\sigma_k$ is well-defined.
\end{example}
\begin{figure}
\centering
\includegraphics[width=0.7\textwidth]{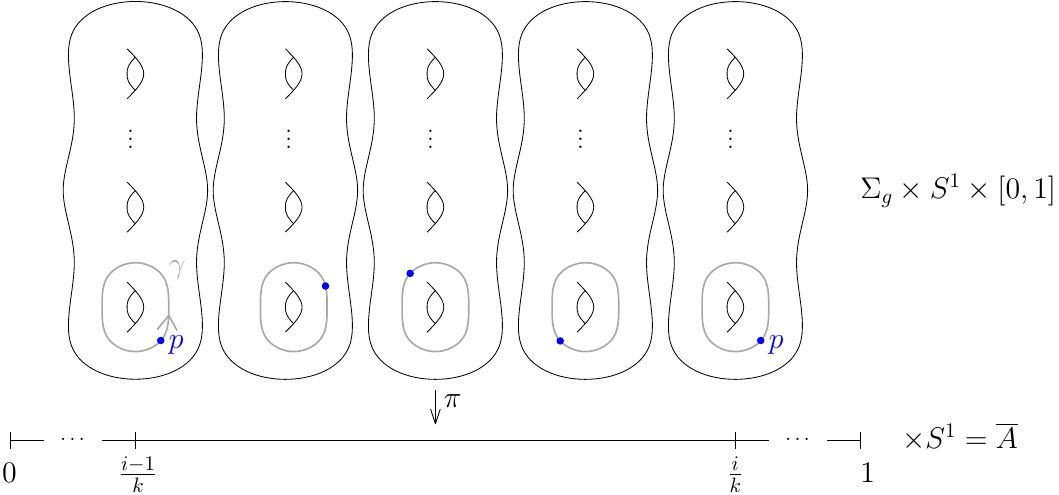}
\caption{Let $k > 0$. The blue point traces out the section $\sigma_k$ over the fibers of $\pi$ over $(\theta, t) \in A$ as $t \in [0, 1]$ varies. This picture is independent of the $S^1$-coordinate $\theta$.}\label{fig:section}
\end{figure}

More generally, consider a Lefschetz fibration $\pi: M \to S^2$ of genus $g \geq 2$. Let\footnote{
  To find $\varphi$, recall that we can always find a representative $\varphi_0$ of $\rho_{(\pi, s)}(C)$
  fixing a neighborhood of $p$. Now the proof of \cite[Prop 1.10]{farb-margalit} shows that
  $\varphi_0(\gamma)$ is isotopic to $\gamma$ rel $p$. Hence, we can isotope (rel $p$) $\varphi_0$ to fix
  $\gamma$ pointwise \cite[Prop 1.11]{farb-margalit}. Finally, $\varphi_0$ can be isotoped to
fix a tubular neighborhood $\alpha$ of $\gamma$ by the uniqueness of tubular neighborhoods up to isotopy.} $\varphi \in \Diff^+(\Sigma_g, p)$ be a representative of $\rho_{(\pi,s)}(C) \in \Mod(\Sigma_{g, 1})$ that fixes pointwise an annular neighborhood $\alpha \subseteq \Sigma_{g}$ of the loop $\gamma: \R/\Z \to \Sigma_g$. Let $M_\varphi$ denote the mapping torus
\[
	M_\varphi = (\Sigma_g \times [0, 1]) / (x,0) \sim (\varphi(x), 1).
\]
There is a diffeomorphism $\Phi: \pi^{-1}(\overline A) \to M_\varphi \times [0, 1]$ such that $\Phi|_{\pi^{-1}(b)} = \Phi_b \times \{b\}$ (for the fixed regular value $b \in \overline A$ of $\pi$) and such that the following diagram commutes:
\begin{figure}[H]
\centering
\begin{tikzcd}
\pi^{-1}(\overline A) \arrow[d, "\pi"] \arrow[rr, "\Phi"] &  & M_\varphi \times [0, 1] \arrow[d, "{(x, \theta, t) \mapsto (\theta, t)}"] \\
\overline A \arrow[rr, Rightarrow, no head] \arrow[u, "s", bend left]    &  & S^1 \times [0, 1] \arrow[u, "{(\theta, t) \mapsto (p, \theta, t)}", bend left]
\end{tikzcd}
\end{figure}

For all $k \in \Z$, let $\sigma_k: M \to S^2$ be defined by
\[
	\sigma_k(x) = \begin{cases}
	s(x) & \text{if }x \in D_1 \cup D_2, \\
	\Phi^{-1}(\gamma(tk), \theta, t) & \text{if }x = (\theta, t) \in \overline A.
	\end{cases}
\]
The map $\sigma_k$ is a section of $\pi$ by construction. To see that $\sigma_k$ is well-defined, it is enough to check this over $\overline A$. We have for any $t \in [0, 1]$ and $x = (0, t) = (1, t)$,
\[
	\sigma_k(0, t) = \Phi^{-1}(\gamma(tk), 0, t) = \Phi^{-1}(\varphi(\gamma(tk)), 1, t) =\Phi^{-1}(\gamma(tk), 1, t) = \sigma_k(1, t).
\]
Moreover, $tk$ is an integer for any $(\theta, t) \in \partial A$, so
\[
	\Phi^{-1}(\gamma(tk), \theta, t) = \Phi^{-1}(p, \theta, t) = s(\theta, t).
\]

\subsection{Monodromy factorizations}\label{sec:monodromy-factorization}

The next proof describes monodromy factorizations of the sections $\sigma_k: S^2 \to M$ using a different description of these sections.
\begin{proof}[Proof of Theorem \ref{thm:infinite-sections}\ref{thm:monodromy-factorization}]
Take a diffeomorphism $h_\gamma \in \Diff^+(\Sigma_{g, 1})$ representing the mapping class $P_\gamma \in \Mod(\Sigma_{g,1})$ and an isotopy $h_\gamma^t: \Sigma_{g} \to \Sigma_g$ with $h_\gamma^0 =\Id_{\Sigma_g}$ and $h_\gamma^1 = h_\gamma$ so that
\begin{enumerate}[(a)]
\item $h_{\gamma}^t$ is supported in the fixed neighborhood $\alpha \subseteq \Sigma_g$ of $\gamma$ for all $t$, and
\item $h_\gamma^t(p) = \gamma(t) \in \Sigma_g$ for all $t \in [0, 1]$.
\end{enumerate}
For any $k \in \Z$, let $h_{\gamma,k}^t: \Sigma_g \to \Sigma_g$ denote the isotopy with $h_{\gamma,k}^0 =\Id_{\Sigma_g}$ and $h_{\gamma,k}^1 = h_\gamma^k$ obtained by concatenating $\lvert k \rvert$-many copies of the isotopy $h_{\gamma}^t$ or $(h_\gamma^t)^{-1}$. Define a diffeomorphism of pairs $p_\gamma: (\partial X_1, s(\partial(D_1 \cup \overline A))) \to (\partial X_2, s(\partial D_2))$ by
\[
	p_\gamma := \Phi^{-1}\circ (h_\gamma \times \Id) \circ \Phi .
\]
Let $N_k := X_1 \cup_{p_\gamma^{k}} X_2$. The naturally defined map $\pi_k: N_k \to S^2$ is a genus-$g$ Lefschetz fibration with a section $S_k: S^2 \to N_k$ obtained by gluing together the restrictions $s|_{X_1}$ and $s|_{X_2}$ accordingly.

The monodromy factorization $(\pi_k, S_k)$ with respect to the identification $\Phi_b: (\pi_k^{-1}(b), S_k(b)) \to \Sigma_{g,1}$ via the identification $\pi_k^{-1}(b) = \pi^{-1}(b) \subseteq \partial X_1$ and a fixed choice of generators $\gamma_i$ of $\pi_1(S^2 - \{q_1, \dots, q_{r_1+r_2}\}, b)$ is
\[
  \left(T_{P_{\gamma}^{-k}(\ell_{r_1+r_2})} \dots T_{P_{\gamma}^{-k}(\ell_{r_1+1})}\right)
  \left(T_{\ell_{r_1}} \dots T_{\ell_1}\right)= 1 \in \Mod(\Sigma_{g,1}),
\]
where we note that the equality of mapping classes holds because the product $T_{\ell_{r_1}} \dots T_{\ell_1}$ commutes with $P_\gamma$ and because $T_{P_\gamma(\ell_i)} = P_\gamma T_{\ell_i} P_\gamma^{-1}$ for any curve $\ell_i$.

On the other hand, one can check that the diffeomorphism $\Psi_k: N_k \to M$ defined by
\[
	\Psi_k(x) := \begin{cases}
	x & \text{ if }x \in \pi^{-1}(D_1 \cup D_2), \\
	\Phi^{-1}(h_{\gamma, k}^t(y), \theta, t) & \text{ if }x = \Phi^{-1}(y, \theta, t) \in \pi^{-1}(A)
	\end{cases}
\]
induces an isomorphism of pairs $(\pi_k, S_k)\cong (\pi, \sigma_k)$ lying over the identity $\Id_{S^2}: S^2 \to S^2$. The restriction $\Psi_k: \pi_k^{-1}(b) = \pi^{-1}(b) \to \pi^{-1}(b)$ is the map $\Phi_b^{-1} \circ h_\gamma^k \circ \Phi_b$, and hence the monodromy factorization with respect to the identification $\Phi_b: (\pi^{-1}(b), \sigma_k(b)) \to \Sigma_{g,1}$ of $(\pi, \sigma_k)$ is the $h_\gamma^k$-conjugate of the factorization of $(\pi_k, S_k)$, i.e.
\[
	\left(T_{\ell_{r_1+r_2}} \dots T_{\ell_{r_1+1}}\right) \left(T_{P_\gamma^{k}(\ell_{r_1})} \dots T_{P_\gamma^{k}(\ell_{1})}\right)  = 1 \in \Mod(\Sigma_{g,1}). \qedhere
\]
\end{proof}

\begin{rmk}[Self-intersection number of $\sigma_k$]\label{rmk:self-intersection}
Consider the monodromy factorization of $(\pi, \sigma_k)$ in $\Mod(\Sigma_{g,1})$ as a factorization instead in $\Mod(\Sigma_g^1)$ of the form
\[
	\left(T_{\ell_{r_1+r_2}} \dots T_{\ell_{r_1+1}}\right) \left(T_{P_\gamma^{k}(\ell_{r_1})} \dots T_{P_\gamma^{k}(\ell_{1})}\right) = T_\delta^{a_k} \in \Mod(\Sigma_g^1)
\]
for some $a_k \in \Z$, where $\delta$ denotes the boundary component of $\Sigma_g^1$. Then the self-intersection $Q_M([\sigma_k(S^2)], [\sigma_k(S^2)]) = -a_k$ (\cite[Lemma 2.3]{smith}). Because any lift of $P_\gamma$ to $\Mod(\Sigma_{g}^1)$ commutes with $T_\delta$ and with $(T_{\ell_{r_1+r_2}} \dots T_{\ell_{r_1+1}})$ in $\Mod(\Sigma_g^1)$, there is an equality of mapping classes
\[
	T_\delta^{a_k} = \left(T_{\ell_{r_1+r_2}} \dots T_{\ell_{r_1+1}}\right) \left(T_{P_\gamma^{k}(\ell_{r_1})} \dots T_{P_\gamma^{k}(\ell_{1})}\right) = (T_{\ell_{r_1+r_2}} \dots T_{\ell_{r_1+1}}) (T_{\ell_{r_1}} \dots T_{\ell_1})  = T_\delta^{a_0} \in \Mod(\Sigma_g^1)
\]
for all $k \in \Z$. So for all $k \in \Z$, the section $\sigma_k$ has self-intersection number $-a_0$.
\end{rmk}

\section{Homology classes of $\sigma_k(S^2)$}\label{sec:proof-main-thm}

The goal of this section is to prove Theorem \ref{thm:infinite-sections}\ref{thm:homology-distinct}. In Section \ref{sec:coords} we record a key computation of signed intersection numbers in $\Sigma_g \times (0, 1)^2$, and use it to prove Theorem \ref{thm:infinite-sections}\ref{thm:homology-distinct} in Section \ref{sec:main-proof}.

\subsection{Signed intersections in $\Sigma_g \times (0, 1)^2$}\label{sec:coords}
Fix the natural inclusion
\[
	\Sigma_g \times (0, 1) \times (0, 1) \subseteq M_\varphi \times (0, 1) = (\Sigma_g \times [0, 1])/\sim) \times (0, 1).
\]
Below, we record the signed intersection number in $\Sigma_g \times (0,1) \times (0, 1)$ of the submanifolds
\[
	\mathcal S_k = \Phi\circ\sigma_k((0, 1) \times (0, 1)) \qquad \text{ and } \qquad \mathcal T = \delta \times \{\theta_0\} \times (0, 1).
\]
Here, $\theta_0 \in (0, 1)$ is the $S^1$-coordinate of the point $b = (\theta_0, 1) \in S^1 \times \{1\} \subseteq \partial A$.

Orient the submanifold $\mathcal S_k \subseteq \Sigma_g \times (0, 1)^2$ via the orientation of $(0, 1) \times (0, 1)$. To orient $\mathcal T$, recall that $\gamma \subseteq \Sigma_g$ is an oriented loop based at $p \in \Sigma_g$ intersecting $\delta$ transversely and exactly once; denote this intersection point by $x \in \Sigma_g$. Fix an orientation on $\delta \subseteq \Sigma_g$ so that $\hat i([\gamma], [\delta]) > 0$, where $\hat i$ denotes the algebraic intersection form of $\Sigma_g$. This in turn defines the product orientation on $\mathcal T$.

\begin{lem}\label{lem:key-comp}
The submanifolds $\mathcal T$ and $\mathcal S_k$ intersect transversely $\lvert k \rvert$-times in $\Sigma_g \times (0,1) \times (0, 1)$ at
\[
	\mathcal T \cap \mathcal S_k = \{(x, \theta_0, t_i) : \gamma(kt_i) = x, \, \, 1 \leq i \leq \lvert k\rvert\}.
\]
The signed intersection number of the oriented submanifolds $\mathcal T$ and $\mathcal S_k$ in $\Sigma_g \times (0, 1) \times (0, 1)$ is $k$.
\end{lem}
\begin{proof}
By definitions of $\mathcal S_k$, the section $\sigma_k$, and the diffeomorphism $\Phi$,
\[
	\mathcal S_k = \Phi \circ \sigma_k((0, 1) \times (0, 1)) = \{(\gamma(kt), \theta, t) : \theta \in (0, 1), \, t \in (0, 1)\}.
\]
Because $\delta$ and $\gamma$ intersect only at the point $x \in \Sigma_g$, the intersection $\mathcal T \cap \mathcal S_k$ is contained in $\{(x, \theta_0)\} \times (0, 1)$. There exist exactly $\lvert k \rvert$-many values $t_1, \dots, t_{\lvert k\rvert} \in (0, 1)$ so that $\gamma(kt_i) = x$. This determines the $\lvert k \rvert$-many points in $\mathcal T \cap \mathcal S_k$ as claimed. It now suffices to show that the sign of each intersection point is $\sign(k)$. We may assume that $k \neq 0$ because $\mathcal T \cap \mathcal S_0 = \emptyset$.

The curve $\widehat{\gamma}: (0, 1) \to \Sigma_g \times (0, 1) \times (0, 1)$ defined as
\[
	\widehat{\gamma}: t \mapsto (\gamma(kt), \theta_0, t)
\]
passes through each point $(x, \theta_0, t_i)$ and is contained in $\mathcal S_k$. See Figure \ref{fig:section-tangent}. The tangent space $T_{(x, \theta_0, t_i)}\mathcal S_k$ contains
\[
	\widehat \gamma'(t_i) = (k \gamma'(kt_i), 0, \partial_t) \in T_x \Sigma_g \times T_{\theta_0} (0, 1) \times T_{t_i} (0, 1),
\]
the tangent vector to $\widehat{\gamma}$ at $(x, \theta_0, t_i)$. Let $\partial_\theta \in T_{\theta_0}(0, 1)$ be a tangent vector so that the vectors $(0, \partial_\theta, 0)$ and $(k \gamma'(kt_i), 0, \partial_t)$ forms a positive basis of $T_{(x, \theta_0, t_i)} \mathcal S_k$.

The sign of the intersection point $(x, \theta_0, t_i)$ is the sign of the ordered basis
\[
	(\partial_\delta, 0, 0), \, (0, 0, \partial_t), \, (0, \partial_{\theta}, 0), \, (k \gamma'(kt_i), 0, \partial_t) \in T_{x} \Sigma_g \times T_{\theta_0}(0,1) \times T_{t_i}(0, 1),
\]
where $\partial_\delta$ denotes a positive tangent vector of $T_{x}\delta$. It suffices to compute the sign of the equivalent ordered basis
\[
	(k \gamma'(kt_i), 0, 0), \,(\partial_\delta, 0, 0),\, (0, \partial_{\theta}, 0), \, (k \gamma'(kt_i), 0, \partial_t).
\]
Because $\hat i([\gamma], [\delta]) > 0$ in $\Sigma_g$, the ordered basis formed by $k \gamma'(kt_i)$ and $\partial_\delta$ of $T_x\Sigma_g$ has sign equal to $\sign(k)$. Because $(0, \partial_\theta,0), (k \gamma'(kt_i), 0, \partial_t)$ forms a positive basis of $T_{(x, \theta_0, t_i)}\mathcal S_k$, and the oriented submanifolds $\Sigma_g$ and $\mathcal S_k$ intersect positively in $\Sigma_g \times (0,1) \times (0, 1)$, the sign of $(x, \theta_0, t_i)$ is also equal to $\sign(k)$.
\end{proof}

\begin{figure}
\centering
\includegraphics[width=0.9\textwidth]{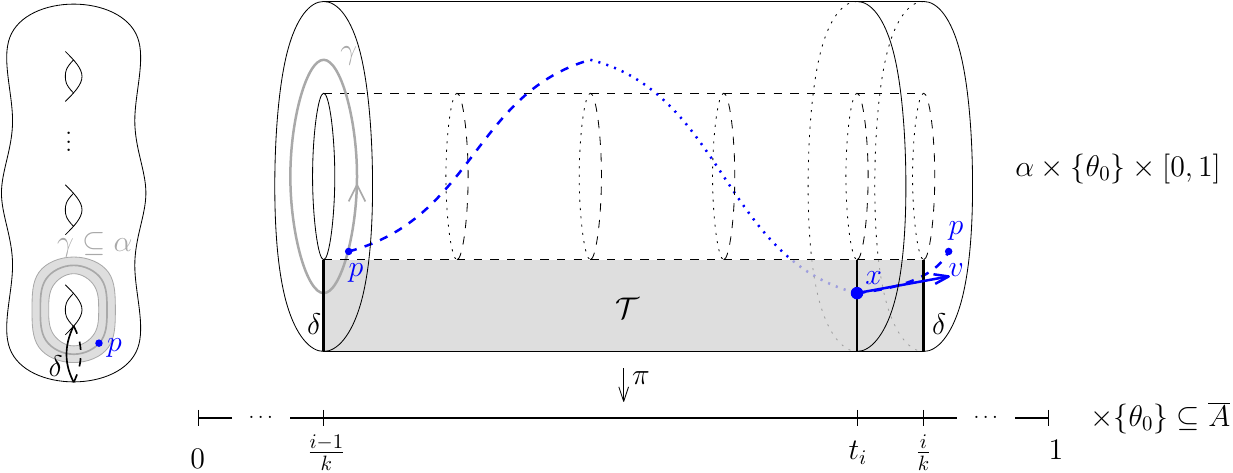}
\caption{Left: The grey annulus in $\Sigma_g$ depicts the neighborhood $\alpha$ of $\gamma$ which is pointwise fixed by the monodromy $\varphi$ over $C$. Middle: The cylinder depicts $\alpha \times \{\theta_0\} \times [0, 1] \subseteq M_\varphi \times [0, 1]$ lying over $\{\theta_0\} \times [0, 1] \subseteq \overline A$. The dotted blue curve depicts the curve $\widehat{\gamma}$ and the vector $v$ denotes its tangent vector $\widehat \gamma'(t_i)$ at the point $(x, \theta_0, t_i)$. The grey rectangle depicts the intersection $\mathcal T \cap (\alpha \times \{\theta_0\}\times [0, 1])$.}\label{fig:section-tangent}
\end{figure}

\subsection{Proof of Theorem \ref{thm:infinite-sections}\ref{thm:homology-distinct}}\label{sec:main-proof}

Recall that $[\delta] = 0 \in H_1(X_i; \Z)$ for both $i = 1, 2$ by assumption in Theorem \ref{thm:infinite-sections} (cf. Remark \ref{rmk:delta-geom}). There exist oriented, $2$-dimensional submanifolds $T_i \subseteq X_i$ with boundary such that
\[
	\Phi(\partial T_1) = \Phi(\overline{\partial T_2}) = \delta \qquad \text{ and } \qquad \Phi(T_1 \cap \pi^{-1}(A)) = \mathcal T \subseteq \Sigma_g \times (0,1) \times (0, 1) \subseteq M_\varphi \times [0, 1].
\]
Here, $\overline{\partial T_2}$ denotes $\partial T_2$ with the opposite orientation. Let $T := T_1 \cup T_2 \subseteq M$, so that $T$ defines a class in $H_2(M; \Z)$. See Figure \ref{fig:T1T2}.
\begin{figure}
\centering
\includegraphics[width=0.6\textwidth]{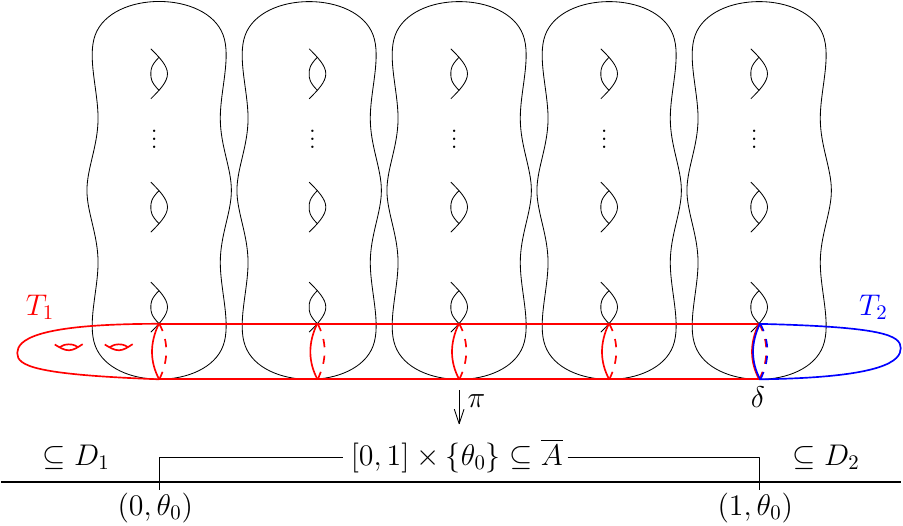}
\caption{The red surface depicts $T_1$. The blue surface depicts $T_2$. The surface $T$ is defined to be the union $T_1 \cup T_2$. If $\delta$ is a vanishing cycle corresponding to a singular value in $D_i$ (for some $i = 1, 2$) then $T_i$ may be chosen to be its thimble contained in $X_i$.}\label{fig:T1T2}
\end{figure}

\begin{lem}\label{lem:intersection}
For all $k \in \Z$,
\[
	Q_{M}([T], [\sigma_k(S^2)] - [\sigma_0(S^2)]) = k.
\]
\end{lem}
\begin{proof}
The homology class $[\sigma_k(S^2)] - [\sigma_0(S^2)] \in H_2(M; \Z)$ is represented by the (oriented) torus formed by gluing two annuli
\[
	\Sigma := \sigma_k(A) \cup \overline{\sigma_0(A)},
\]
where we recall that by construction, $\sigma_k|_{\partial A}  = \sigma_0|_{\partial A}$ and $\overline{\sigma_0(A)}$ denotes $\sigma_0(A)$ with the opposite orientation. By perturbing $\sigma_0(A)$ and $\sigma_k(A)$ in an open neighborhood of $\pi^{-1}(A)$, we may arrange so that $\Sigma$ is a smoothly embedded torus and so that $\sigma_0(A)$ is disjoint from $T$. See Figure \ref{fig:section-torus}.
\begin{figure}
\centering
\includegraphics[width=0.45\textwidth]{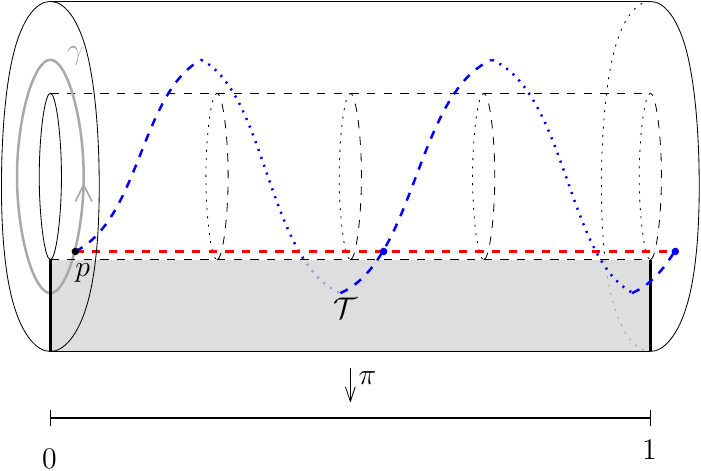} \qquad\qquad
\includegraphics[width=0.45\textwidth]{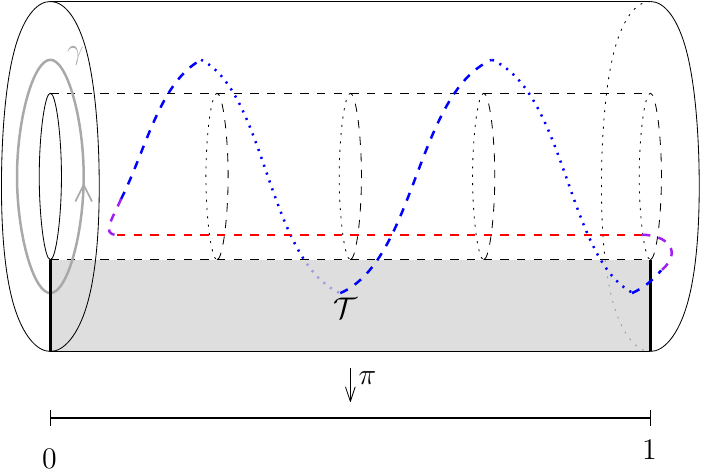}
\caption{The case of $k = 2$. Both cylinders depict $\alpha \times \{\theta_0\} \times [0, 1]$ lying over $\{\theta_0\} \times [0, 1] \subseteq \overline A$ as in Figure \ref{fig:section-tangent}. Left: The red and blue dotted curves depict $\mathcal S_0$ and $\mathcal S_k$ respectively. Right: The curves $\mathcal S_0$ and $\mathcal S_k$ are perturbed in the cylinder $\alpha \times \{\theta_0\} \times [0, 1]$ so that their union $\beta$ forms a smooth circle. The product $\Phi^{-1}(\beta \times S^1) \subseteq \Phi^{-1}(\alpha \times S^1 \times [0, 1])$ is the torus $\Sigma$ representing the class $[\sigma_1(S^2)] - [\sigma_0(S^2)]$.}\label{fig:section-torus}
\end{figure}

Because $\Sigma$ is contained in $\pi^{-1}(A)$, the intersection number $Q_{M}([T], [\Sigma])$ is equal to the signed intersection number between $T_1$ and $\sigma_k(A)$. It suffices to compute the signed intersection number of $\Phi(T_1 \cap \pi^{-1}(A))$ and $\Phi(\sigma_k(A))$ in $M_\varphi \times [0, 1]$ because $\Sigma \subseteq \pi_1^{-1}(A)$ and $\Phi$ is orientation-preserving. Lemma \ref{lem:key-comp} implies that $Q_{M}([T], [\Sigma]) = k$.
\end{proof}

\begin{proof}[Proof of Theorem \ref{thm:infinite-sections}\ref{thm:homology-distinct}]
Suppose that $[\sigma_{k_1}(S^2)] = [\sigma_{k_2}(S^2)]$ in $H_2(M; \Z)$ for some $k_1, k_2 \in \Z$. By Lemma \ref{lem:intersection},
\[
	k_1-k_2 = Q_{M}([T], [\sigma_{k_1}(S^2)] - [\sigma_{k_2}(S^2)]) = 0. \qedhere
\]
\end{proof}

\section{Compatibility with the Gompf--Thurston symplectic form}\label{sec:symp}

The goal of this section is to prove Theorem \ref{thm:infinite-sections}\ref{thm:symplectic-form} by appropriately modifying the Gompf--Thurston construction to produce a symplectic form $\omega$ for which each $\sigma_k(S^2)$ is symplectic. To do so, we closely follow Gompf's argument \cite[Theorem 10.2.18]{gompf-stipsicz}.
Although it may be immediate to the experts, we record how to extract the desired conclusions from Gompf's proofs for the sake of completeness.
The main observation is that the sections $\sigma_k$ differ
only over the annulus $A$. Moreover, the images of all $\sigma_k|_A$
lie on $\gamma \times A$. So, we only need
to control the behavior of a chosen symplectic form over a neighborhood
of $\gamma \times A$.

\begin{proof}[{Proof of Theorem \ref{thm:infinite-sections}\ref{thm:symplectic-form}}]
  Throughout this proof, $H^*$ denotes de Rham cohomology. Let $K \subseteq M$ denote the set of critical points of $\pi$. Let $e \in H^2(M-K; \R) \cong H^2(M; \R)$ denote the Euler class of the oriented $2$-plane bundle $TF \to (M-K)$ of tangent vectors to the fibers of $\pi$.
For any regular fiber $F_y \subseteq M$, the class $(2-2g)^{-1} e \in H^2(M; \R)$ satisfies
\[
	\langle (2-2g)^{-1} e, [F_y] \rangle = 1,
\]
because the restriction of $e$ to $F_y$ is the Euler class of $TF_y \to F_y$. For any closed, smooth surface $S \subseteq M$ of genus $h$ contained in a singular fiber of $\pi$, one can prove using the Poincar\'e--Hopf theorem that $\langle e, [S] \rangle = 2-2h-1$ and hence $\langle (2-2g)^{-1}e, [S] \rangle > 0$. (Recall that by assumption, no singular fiber of $\pi: M \to S^2$ contains an embedded sphere.) Let $\zeta_0$ be a closed $2$-form of $M$ so that \[[\zeta_0] = (2-2g)^{-1}e \in H^2(M; \R).\]

\medskip\noindent
\emph{Step 1: Constructing the $2$-forms $\eta_i$ and the $1$-forms $\beta_i$.} By the proof of \cite[Theorem 10.2.18]{gompf-stipsicz}, there exists a closed $2$-form $\eta_0$ on $M$ with $[\eta_0] = [\zeta_0]$ such that
\begin{enumerate}[(1)]
	\item $\eta_0|_{S}$ is symplectic for any closed surface $S$ contained in a fiber of $\pi$, and
	\item for any $p \in K$ and for some charts $U_p \subseteq \C^2$ and $V_{\pi(p)} \subseteq \C$ on which $\pi$ takes the form $(z_1, z_2) \mapsto z_1^2 + z_2^2$, the restriction $\eta_0|_{U_p}$ is the standard symplectic form $dx_1 \wedge dy_1 + dx_2 \wedge dy_2$.
\end{enumerate}
Also see the proof of \cite[Theorem 2.7]{gompf-symp} for a description of $\eta_0$ around the points of $K$.

We will modify $\eta_0$ by cohomologous $2$-forms on an open cover of $\pi^{-1}(A)$. Let $A' \subseteq A'' \subseteq S^2$ be small open neighborhoods of $A$ containing no critical values of $\pi$ so that $\overline{A} \subseteq A'$ and $\overline{A'} \subseteq A'' = S^1 \times (-\varepsilon, 1 + \varepsilon)$ for some $\varepsilon > 0$. Extend the identification $\Phi$ to $\pi^{-1}(A'') \to M_\varphi \times (-\varepsilon, 1 + \varepsilon)$. Let $A'' = V_1 \cup V_2$ where $V_i$ is an open ball for each $i = 1, 2$.

Fix a unit-area symplectic form $\omega_{\Sigma_g}$ of $\Sigma_g$. Recall that $\varphi$ is chosen so that $\rho_{(\pi, s)}(C) = [\varphi]$ and so that $\varphi$ fixes $\alpha$ pointwise. By Moser's trick, there is a smooth isotopy $\varphi_t: \Sigma_g \to \Sigma_g$ for $t \in [0, 1]$ with $\varphi_0^*\omega_{\Sigma_g} = \omega_{\Sigma_g}$ and $\varphi_1 = \varphi$. By a relative form of Moser's trick \cite[Exercise 3.2.6]{mcduff-salamon}, we may arrange that $\varphi_t$ fixes $\alpha$ pointwise for all $t \in [0, 1]$, after possibly shrinking $\alpha$. Let $\theta_0'$ be the pullback of $\omega_{\Sigma_g}$ on $\Sigma_g \times \R$ via the projection $\Sigma_g \times \R \to \Sigma_g$. Then $\theta_0'$ induces a $2$-form $\theta_0$ on $M_{\varphi_0}$ via the covering map $\Sigma_g \times \R \to M_{\varphi_0}$ because $\theta_0'$ is invariant under deck transformations. 

Because there is an isotopy $\varphi_t$ from $\varphi_0$ to $\varphi$ that fixes $\alpha$ pointwise for every $t$, there is an isomorphism $F: M_{\varphi} \to M_{\varphi_0}$ of $\Sigma_g$-bundles over $S^1$ so that $F$ maps $\gamma \times S^1\subseteq M_{\varphi}$ diffeomorphically onto $\gamma \times S^1 \subseteq M_{\varphi_0}$. The restriction of the $2$-form $\theta := F^*\theta_0$ of $M_{\varphi}$ to $\gamma \times S^1 \subseteq M_\varphi$ satisfies
\[
	\theta|_{\gamma \times S^1 \subseteq M_\varphi} = 0. 
\]

Consider the composition
\[
	\Phi_0: \pi^{-1}(A'') \xrightarrow{\Phi} M_{\varphi} \times (-\varepsilon , 1 + \varepsilon) \xrightarrow{\mathrm{pr}_1} M_{\varphi}
\]
so that $\Phi_0^*\theta$ is a closed $2$-form on $\pi^{-1}(A'')$ that is symplectic along the fibers of $\pi$. On $\pi^{-1}(V_i)$ for each $i = 1, 2$, define the $2$-form
\[
	\eta_i := \Phi_0^*\theta|_{\pi^{-1}(V_i)}.
\]

Now we construct a $1$-form $\beta_i$ on $\pi^{-1}(V_i)$. Recall that $\alpha \subseteq \Sigma_g$ is an annular neighborhood of $\gamma$ fixed pointwise by $\varphi$. By possibly shrinking $\alpha$, let $\alpha \subseteq \alpha' \subseteq \alpha'' \subseteq \Sigma_g$ denote open, annular neighborhoods of $\gamma$ that are pointwise fixed by $\varphi$ so that $\overline{\alpha} \subseteq \alpha'$ and $\overline{\alpha'}\subseteq \alpha''$. There are submanifolds $T$ and $N''$ of $M_\varphi \times (-\varepsilon, 1 + \varepsilon)$, where
\[
	N'' := \alpha'' \times A'' \simeq \gamma \times S^1 \times \{0\}  =: T.
\]
Note that $\Phi_0^*\theta|_{\Phi^{-1}(T)} = 0$, and so $\Phi_0^*\theta|_{\Phi^{-1}(N'')}$ is an exact $2$-form. As $TF|_{\Phi^{-1}(N'')}$ is a trivial bundle, $\zeta_0|_{\Phi^{-1}(N'')}$ is also exact.

By modifying $\zeta_0$ only over $\Phi^{-1}(N'')$, we may find a closed $2$-form $\zeta$ of $M$ so that $[\zeta] = [\zeta_0] \in H^2(M; \R)$ and $\zeta|_{\Phi^{-1}(N')} = \Phi_0^*\theta|_{\Phi^{-1}(N')}$ where $N' := \alpha' \times A'$ is an open submanifold of $N''$. Because $\pi^{-1}(V_i) \simeq \Sigma_g$ and $\langle [\zeta|_{\pi^{-1}(V_i)}], [F_y] \rangle =\langle [\Phi_0^*\theta|_{\pi^{-1}(V_i)}], [F_y]\rangle = 1$ for any regular fiber $F_y$,
\[
	\Phi_0^*\theta|_{\pi^{-1}(V_i)}  - \zeta|_{\pi^{-1}(V_i)} = d\beta_i'
\]
for some $1$-form $\beta_i'$ of $\pi^{-1}(V_i)$. Note that $\beta_i'|_{N_i'}$ is a closed $1$-form restricted to $N_i' := \pi^{-1}(V_i) \cap \Phi^{-1}(N')$.
As $N_i' \simeq \gamma$, there is some closed $1$-form $\tau_i$ on $\pi^{-1}(V_i)$ so that $[\beta_i'|_{N_i'} - \tau_i|_{N_i'}] = 0$.
Thus, by modifying the $1$-form $\beta_i' - \tau_i$ only over $N_i'$, we may obtain a $1$-form $\beta_i$ of $\pi^{-1}(V_i)$ so that
$d\beta_i = d\beta_i'$ and $\beta_i|_{N_i} = 0$, where $N_i:= (\alpha \times A) \cap \pi^{-1}(V_i)$ is an open submanifold of $N_i'$.

\medskip\noindent
\emph{Step 2: Constructing the symplectic form $\omega_{t}$.}
Let $B\subseteq S^2$ be an open set such that $S^2 = A'' \cup B$ and $B \cap A = \emptyset$. There exists a $1$-form $\beta_0$ on $\pi^{-1}(B)$ so that
\[
	\eta_0|_{\pi^{-1}(B)} = \zeta|_{\pi^{-1}(B)} + d\beta_0.
\]
Let $\{\rho_0, \rho_1, \rho_2\}$ be a partition of unity subordinate to the open cover $S^2 = B \cup V_1 \cup V_2$ and consider the $2$-form
\[
	\eta := \zeta + \sum_{i=0}^2 d((\rho_i \circ \pi)\beta_i) = \sum_{i=0}^2 (\rho_i \circ \pi) \eta_i + \sum_{i=0}^2 d(\rho_i \circ \pi) \wedge \beta_i.
\]
The form $\eta$ is symplectic along the smooth loci of the fibers of $\pi$. For any symplectic form $\omega_{S^2}$ of $S^2$ and for any $t > 0$ small enough, the closed $2$-form
\[
	\omega_t := t \eta + \pi^* \omega_{S^2}
\]
is symplectic on $M$ and $\sigma_0(S^2)$ is a symplectic submanifold of $(M, \omega_t)$ by compactness \cite[Proposition 10.2.20]{gompf-stipsicz}. The smooth locus of any fiber of $\pi$ is a symplectic submanifold of $(M, \omega_t)$ by construction of $\eta$.

\medskip\noindent
\emph{Step 3: Restricting to $\sigma_k(S^2)$.} Let $k \neq 0 \in \Z$. The sections $\sigma_0$ and $\sigma_k$ agree on a neighborhood $B_0$ of $S^2 - A$, so the restriction $\omega_{t}|_{\sigma_k(B_0)}$ is symplectic. Recall that $\sigma_k(A) \subseteq N_1 \cup N_2$ and compute that the restriction $\eta|_{N_1 \cup N_2}$ is
\[
	\eta|_{N_1 \cup N_2} = \sum_{i=1}^2 (\rho_i \circ \pi) \eta_i|_{N_i} + \sum_{i=1}^2d(\rho_i \circ \pi)|_{N_i} \wedge \beta_i|_{N_i} =  \sum_{i=1}^2 (\rho_i \circ \pi) \eta_i|_{N_i} = \Phi_0^*\theta|_{N_1 \cup N_2}
\]
by construction of the forms $\beta_i$ and $\eta_i$. The form $\pi^*(\omega_{S^2})$ restricts to a symplectic form on $\sigma_k(A) \subseteq \pi^{-1}(A)$. On the other hand, $\Phi_0(\sigma_k(A)) = \gamma \times S^1 \subseteq M_\varphi$. Because $\theta|_{\gamma \times S^1} = 0$, the restriction $\omega_t|_{\sigma_k(A)}$ is symplectic.
\end{proof}

\section{Proof of Theorem \ref{thm:cor-infty} and other applications of Theorem \ref{thm:infinite-sections}}\label{sec:apps}
The goal of this section is to prove some corollaries of Theorem \ref{thm:infinite-sections}. In Section \ref{sec:cor-infty} we prove that certain fiber sums of Lefschetz fibrations admit infinitely many homologically distinct sections and deduce Theorem \ref{thm:cor-infty}. In Section \ref{sec:indec} we provide an explicit example of a genus-$g$, fiber sum indecomposable Lefschetz fibration satisfying the assumptions of Theorem \ref{thm:infinite-sections} to prove Corollary \ref{cor:intro-indec}.

\subsection{Proof of Theorem \ref{thm:cor-infty} and other corollaries}\label{sec:cor-infty}

Before proving Theorem \ref{thm:cor-infty}, we record some corollaries of Theorem \ref{thm:infinite-sections}.
\begin{cor}\label{cor:untwisted}
Let $\pi: M \to S^2$ be a nontrivial Lefschetz fibration of genus $g \geq 2$ admitting a section $s: S^2 \to M$. Suppose that $T_{\ell_r} \dots T_{\ell_1} = 1 \in \Mod(\Sigma_{g,1})$ is a monodromy factorization of $(\pi, s)$. Fix a loop $\gamma \in \pi_1(\Sigma_g, p)$ such that $\gamma$ intersects some vanishing cycle $\ell_i$ transversely once in $\Sigma_g$. The untwisted fiber sum $\pi \#_F \pi: M \#_F M \to S^2$ admits infinitely many homologically distinct smooth sections $\{\sigma_k: S^2 \to M \#_F M: k \in \Z\}$. A monodromy factorization for $(\pi \#_F \pi, \sigma_k)$ is
\[
	\left(T_{\ell_r} \dots T_{\ell_1}\right)\left(T_{P_\gamma^{k}(\ell_r)}\dots T_{P_\gamma^{k}(\ell_1)}\right) = 1 \in \Mod(\Sigma_{g,1}).
\]
Furthermore, $M \#_F M$ admits a symplectic structure for which the fibers of $\pi \#_F \pi$ and any section $\sigma_k(S^2)$ are all symplectic.
\end{cor}
\begin{proof}
The section $s \#_F s: S^2 \to M \#_F M$ of $\pi \#_F \pi: M \#_F M \to S^2$ has a monodromy factorization
\[
	\left(T_{\ell_r} \dots T_{\ell_1}\right)\left(T_{\ell_r} \dots T_{\ell_1}\right) = 1 \in \Mod(\Sigma_{g,1}).
\]
Let $\delta = \ell_i \subseteq \Sigma_{g,1}$. Then $\delta \subseteq \Sigma_g$ is nonseparating because $\hat i([\delta], [\gamma]) = \pm 1$, and hence $[\delta] \neq 0 \in H_1(\Sigma_g; \Z)$. Theorem \ref{thm:infinite-sections} applies to $\pi \#_F \pi: M \#_F M\to S^2$ with the subword $T_{\ell_r} \dots T_{\ell_1}$ and the loop $\gamma \in \pi_1(\Sigma_g, p)$ intersecting $\delta = \ell_i$ transversely once in $\Sigma_g$.
\end{proof}

\begin{rmk}
Any nontrivial Lefschetz fibration $\pi: M \to S^2$ admits a nonseparating vanishing cycle $\ell$ by a theorem of Stipsicz \cite[Theorem 1.3]{stipsicz-vanishing-cycles}. There exists a loop $\gamma \in \pi_1(\Sigma_g, p)$ intersecting $\ell$ transversely once by the change-of-coordinates principle \cite[Section 1.3]{farb-margalit}, so Corollary \ref{cor:untwisted} implies Corollary \ref{cor:intro-untwisted}.
\end{rmk}

\begin{cor}\label{cor:trivial-homology}
For $i = 1, 2$, let $\pi_i: M_i \to S^2$ be a Lefschetz fibration of genus $g \geq 2$ that admits  a section $s_i$. Suppose that $H_1(M_i; \Z) = 0$ for each $i = 1, 2$. Then any fiber sum $\pi_1 \#_{F,\psi} \pi_2: M_1\#_{F,\psi} M_2 \to S^2$ admits infinitely many homologically distinct sections $\{\sigma_k: M_1 \#_{F, \psi}M_2 \to S^2: k \in \Z\}$ and a symplectic structure for which the fibers of $\pi_1 \#_{F,\psi}\pi_2$ and any section $\sigma_k(S^2)$ are all symplectic.
\end{cor}
\begin{proof}
  Recall the construction in Example \ref{ex:untwisted}. The monodromy factorization of $(\pi_1 \#_{F,\psi} \pi_2, s_1 \#_{F, \psi} s_2)$ is given by
  \[ (T_{\ell_{r_1 + r_2}}\ldots T_{\ell_{r_1 + 1}}) (T_{\ell_{r_1}} \ldots T_{\ell_{1}})  = 1 \in \Mod(\Sigma_{g,1}), \]
  so that $\{\ell_1,\ldots, \ell_{r_1}\}$ is a set of vanishing cycles of $M_1$ and $\{\ell_{r_1+ 1},\ldots, \ell_{r_1 + r_2} \}$ is a set of vanishing cycles of $M_2$.
  The vanishing cycles of any Lefschetz fibration with at least one section span the kernel of
  the map $H_1(\Sigma_g;\Z) \to H_1(M;\Z)$ induced by inclusion of the regular fiber $\pi^{-1}(b) \cong \Sigma_g$ (e.g. by Mayer--Vietoris and Remark \ref{rmk:delta-geom}). As $H_1(M_i;\Z) = 0$ for $i=1,2$,
  \[ \Z\{[\ell_1],\ldots, [\ell_{r_1}]\} = \Z\{[\ell_{r_1+1}],\ldots, [\ell_{r_1 + r_2}]\} = H_1(\Sigma_g;\Z). \]
  Apply \Cref{thm:infinite-sections} with $\delta = \ell_i$, for any nonseparating vanishing cycle $\ell_i$,
    and $\gamma$ any loop intersecting $\ell_i$ once, which exists by the change-of-coordinates principle~\cite[Section 1.3]{farb-margalit}.
\end{proof}

Theorem \ref{thm:cor-infty} now follows from Corollary \ref{cor:untwisted}.
\begin{proof}[Proof of Theorem \ref{thm:cor-infty}]
For any $g \geq 2$, consider a chain of simple closed curves $c_1, \dots, c_{2g}$ in $\Sigma_g$, so that $c_i$ and $c_{i+1}$ intersect transversely once and $c_i \cap c_j = \emptyset$ if $\lvert i - j \rvert > 2$. By the $(2g)$-chain relation \cite[Proposition 4.12]{farb-margalit}
\[
	(T_{c_1}\dots T_{c_{2g}} )^{4g+2} = 1 \in \Mod(\Sigma_{g,1}).
\]
Let $\pi: M \to S^2$ and $s: S^2 \to M$ be a genus-$g$ Lefschetz fibration and section corresponding to this positive factorization. Applying Corollary \ref{cor:untwisted} to the pair $(\pi, s)$ shows that the untwisted fiber sum $\pi \#_F \pi:  M \#_F M \to S^2$ admits infinitely many homologically distinct sections $\sigma_k$ and $M \#_F M$ admits a symplectic structure for which the fibers of $\pi \#_F \pi$ and any section $\sigma_k(S^2)$ are symplectic.
\end{proof}

\begin{rmk}[Applications to surface bundles and general Lefschetz fibrations]\label{rmk:surface-bundle}
  Similar constructions and arguments can be used to show that certain genus-$g$ Lefschetz fibrations $\pi: M^4 \to \Sigma_h$ (admitting at least one section $s: \Sigma_h \to M$ and possibly with no critical points) have infinitely many homologically distinct sections. For example, let $h= h_1+h_2$ with $h_1, h_2 > 0$ and suppose that for some separating curve $C \subseteq \Sigma_{h_1+h_2} \cong \Sigma_{h_1} \# \Sigma_{h_2}$ along which the connected sum is formed, there exists some $\gamma \in \pi_1(\Sigma_g, p)$ that is fixed by the monodromy $\rho_{(\pi, s)}(C) \in \Mod(\Sigma_{g,1})$ along $C$. Let $A \subseteq \Sigma_{h_1+h_2}$ be an annular neighborhood of $C$. One can modify the section $s$ of $\pi$ over the annulus $A$ as in Section \ref{sec:sections} by twisting along $\gamma \in \pi_1(\Sigma_g, p)$ to create new sections $\sigma_k: \Sigma_{h_1+h_2} \to M$ for each $k \in \Z$. If there exists some $\delta \subseteq \Sigma_{g,1}$ that intersects $\gamma$ once transversely in $\Sigma_g$ and vanishes in $H_1(X_1; \Z)$ and $H_1(X_2; \Z)$ where $X_1$ and $X_2$ are the two connected components of $M - \pi^{-1}(C)$ then the arguments of Section \ref{sec:main-proof} apply directly to show that the sections $\sigma_k$ are pairwise homologically distinct. The Gompf--Thurston construction of Section \ref{sec:symp} also works for Lefschetz fibrations over arbitrary bases $\Sigma_h$, $h \geq 0$, possibly with no critical points, to show that the sections $\sigma_k$ can be made to be symplectic as well. 
\end{rmk}

\subsection{Indecomposable examples}\label{sec:indec}

A Lefschetz fibration is called \emph{indecomposable} if it cannot be written as a fiber sum of two nontrivial Lefschetz fibrations. In this section we provide an example of a genus-$g$, indecomposable Lefschetz fibration admitting infinitely many, homologically distinct symplectic sections for every $g \geq 2$.

\begin{figure}
\centering
\includegraphics[width=0.8\textwidth]{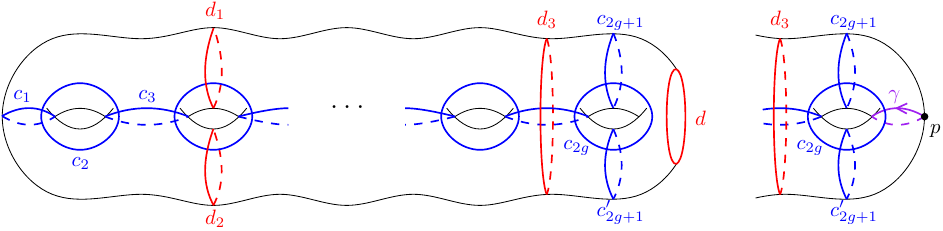}
\caption{The curves that are used in the construction of $(\pi: M_g \to S^2, s: S^2 \to M_g)$. Left: The red curve $d$ is the boundary component of $\Sigma_g^1$. Right: The purple loop $\gamma \in \pi_1(\Sigma_{g}, p)$ is based at the marked point $p$ of $\Sigma_{g,1}$ which is the result of capping $\Sigma_g^1$ by a punctured disk.}\label{fig:indec-curves}
\end{figure}

Consider the curves of $\Sigma_g^1$ depicted in Figure \ref{fig:indec-curves}. The $(2g)$-chain relation \cite[Proposition 4.12]{farb-margalit} gives a factorization of $T_d^2 \in \Mod(\Sigma_g^1)$ by positive Dehn twists
\[
	T_d^2 =(T_{c_1} \dots T_{c_{2g}})^{4g+2}T_d \in \Mod(\Sigma_g^1). 
\]
In what follows, we repeatedly apply the identity (for any $T_\ell$ and $f \in \Mod(\Sigma_g^1)$)
\begin{equation}\label{eqn:moving-twists}
	T_\ell f = f T_{f^{-1}(\ell)} \in \Mod(\Sigma_g^1)
\end{equation}
to obtain new factorizations of $T_d^2$ into positive Dehn twists. By noting that $T_d$ commutes with each twist $T_{c_i}$, $1 \leq i \leq 2g$, and by applying (\ref{eqn:moving-twists}) with $\ell = c_{2g-1}, c_{2g}$ and $f = (T_{c_1} \dots T_{c_{2g-2}})^i$ for $1 \leq i \leq 4g-2$, obtain the factorizations
\begin{align*}
	T_d^2 &= (T_{c_1} \dots T_{c_{2g}})^{4}(T_{c_1} \dots T_{c_{2g}})^{4g-3}(T_{c_1} \dots T_{c_{2g-2}})T_d (T_{c_{2g-1}}T_{c_{2g}}) \\
	 &= (T_{c_1} \dots T_{c_{2g}})^{4}(T_{c_1} \dots T_{c_{2g-2}})^{4g-2} T_d( W T_{c_{2g-1}}T_{c_{2g}}).
\end{align*}
Above, $W$ is a product $2(4g-3)$-many positive Dehn twists about nonseparating curves, each of which is contained in the subgroup $\langle T_{c_i} : 1 \leq i \leq 2g \rangle \leq \Mod(\Sigma_g^1)$. 

By $(2g-2)$- and $3$-chain relations in $\Mod(\Sigma_g^1)$, there is an equality of mapping classes
\begin{align*}
(T_{c_1} \dots T_{c_{2g-2}})^{4g-2} T_d = T_{d_3} T_d = (T_{c_{2g+1}'} T_{c_{2g}} T_{c_{2g+1}})^4.
\end{align*}
Applying this substitution to the factorization of $T_d^2$ obtained above yields a new factorization
\[
T_d^2 = (T_{c_1} \dots T_{c_{2g}})^{4}(T_{c_{2g+1}'} T_{c_{2g}} T_{c_{2g+1}})^4( W T_{c_{2g-1}}T_{c_{2g}}) \in \Mod(\Sigma_g^1).
\]
By applying (\ref{eqn:moving-twists}) repeatedly with $\ell = c_1, c_2, c_3$ and $f = (T_{c_4} \dots T_{c_{2g}})^i$ for $1 \leq i \leq 3$, obtain the factorization
\[
	T_d^2 = (T_{c_1} T_{c_2} T_{c_3})^4 V (T_{c_{2g+1}'} T_{c_{2g}} T_{c_{2g+1}})^4 (W T_{c_{2g-1}} T_{c_{2g}}) \in \Mod(\Sigma_g^1)
\]
where $V$ is a product of $(8g-12)$-many positive Dehn twists about nonseparating curves, each of which is contained in the subgroup $\langle T_{c_i} : 1 \leq i \leq 2g \rangle \leq \Mod(\Sigma_g^1)$. Because $WT_{c_{2g-1}}T_{c_{2g}}$ commutes with $T_d$, we may rearrange so that
\begin{equation}\label{eqn:indec-1}
	T_d^2 = (W T_{c_{2g-1}} T_{c_{2g}})(T_{c_1} T_{c_2} T_{c_3})^4 V (T_{c_{2g+1}'} T_{c_{2g}} T_{c_{2g+1}})^4 \in \Mod(\Sigma_g^1)
\end{equation}
By a $3$-chain relation,
\[
	(T_{c_1} T_{c_2} T_{c_3})^4 = T_{d_1} T_{d_2} \in \Mod(\Sigma_g^1),
\]
which applied to (\ref{eqn:indec-1}) gives the factorization
\begin{equation}\label{eqn:indec-2}
	T_d^2 = (W T_{c_{2g-1}} T_{c_{2g}})(T_{d_1} T_{d_2}) V (T_{c_{2g+1}'} T_{c_{2g}} T_{c_{2g+1}})^4 \in \Mod(\Sigma_g^1).
\end{equation}
Let $(\pi: M_g \to S^2, s: S^2 \to M_g)$ denote the Lefschetz fibration and $(-2)$-section (\cite[Lemma 2.3]{smith}) corresponding to (\ref{eqn:indec-2}).

\begin{lem}\label{lem:Mg-invariants}
The $4$-manifold $M_g$ has Euler characteristic and signature
\[
	\chi(M_g) = 12g + 2, \qquad \sigma(M_g) = -8g-2.
\]
\end{lem}
\begin{proof}
The factorization (\ref{eqn:indec-2}) is a product of $(16g -2)$-many positive Dehn twists. The Euler characteristic of a genus-$g$ Lefschetz fibration with $n$-many vanishing cycles is $4 - 4g + n$; letting $n = 16g -2$ shows that $\chi(M_g) = 12g + 2$.

Next, we compute the signature $\sigma(M_g')$ of an auxiliary Lefschetz fibration $\pi': M_g' \to S^2$ defined by (\ref{eqn:indec-1}) viewed as a positive factorization of the identity in $\Mod(\Sigma_g)$. This factorization is a product of $(16g+ 8)$-many Dehn twists about nonseparating curves, each of which lies in the subgroup $\langle T_{c_i} : 1 \leq i \leq 2g \rangle \leq \Mod(\Sigma_g)$. In particular, $\pi: M_g' \to S^2$ is a hyperelliptic Lefschetz fibation, and Endo's signature formula for hyperelliptic Lefschetz fibrations (\cite[Theorems 4.4(2), 4.8]{endo}) shows that
\[
	\sigma(M_g') = (16g+ 8) \left(-\frac{g+1}{2g+1}\right) = -8g-8.
\]

Using the terminology of Endo--Nagami (\cite{endo-nagami}), the factorization (\ref{eqn:indec-2}) is obtained from (\ref{eqn:indec-1}) via a $\rho$-substitution in $\Mod(\Sigma_g)$ with
\[
	\rho = (T_{c_1} T_{c_2} T_{c_3})^{-4}(T_{d_1}T_{d_2}) \in \Mod(\Sigma_{g})
\]
which has signature $I_g(\rho) = 6$ by \cite[Lemma 3.5(1), Proposition 3.10]{endo-nagami}. By \cite[Theorem 4.3]{endo-nagami},
\[
	\sigma(M_g) = \sigma(M_g') + I_g(\rho) = -8g - 2. \qedhere
\]
\end{proof}

The next proposition deduces that $M_g$ must be minimal from Lemma \ref{lem:Mg-invariants} to conclude that $\pi: M_g \to S^2$ is indecomposable.
\begin{prop}\label{prop:Mg-is-indec}
The Lefschetz fibration $\pi: M_g \to S^2$ is indecomposable.
\end{prop}
\begin{proof}
Lemma \ref{lem:Mg-invariants} shows that $c_1^2(M_g) = 2\chi(M_g) + 3\sigma(M_g) = -2$. It also shows that $b^+(M_g) \geq 2$ because
\[
	b^+(M_g) = \frac{(b^+(M_g) + b^-(M_g)) + \sigma(M_g)}{2} \geq \frac{(\chi(M_g) - 2) + \sigma(M_g)}{2} = 2g-1.
\]
A theorem of Taubes (\cite[Theorem 0.2(3)]{taubes}) implies that $M_g$ equipped with any symplectic form contains an embedded, symplectic $(-1)$-sphere. Usher (\cite[Corollary 1.2]{usher}) showed that any such Lefschetz fibration $\pi: M_g \to S^2$ is indecomposable.
\end{proof}

Finally, we apply Theorem \ref{thm:infinite-sections} to the pair $(\pi: M_g \to S^2, s: S^2 \to M_g)$.
\begin{proof}[Proof of Corollary \ref{cor:intro-indec}]
Take the monodromy factorization (\ref{eqn:indec-2}) of the pair $(\pi, s)$ consider the subword
\[
	(T_{c_{2g+1}'} T_{c_{2g}} T_{c_{2g+1}})^4.
\]
Let $\gamma \in \pi_1(\Sigma_g, p)$ be as shown in Figure \ref{fig:indec-curves}. Then
\[
	(T_{c_{2g+1}'} T_{c_{2g}} T_{c_{2g+1}})^4(\gamma) = T_{d_3}(\gamma)= \gamma \in \pi_1(\Sigma_g, p)
\]
by the $3$-chain relation. Let $\delta = c_{2g}$ so that the Dehn twist $T_{c_{2g}}$ appears in both subwords $(T_{c_{2g+1}'} T_{c_{2g}} T_{c_{2g+1}})^4$ and $(W T_{c_{2g-1}} T_{c_{2g}} )(T_{d_1} T_{d_2}) V$ of (\ref{eqn:indec-2}). Theorem \ref{thm:infinite-sections} now applies to the pair $(\pi: M_g \to S^2, s: S^2 \to M_g)$ with these choices of monodromy subword and curves $\gamma$ and $\delta$. Finally, $\pi: M_g \to S^2$ is indecomposable by Proposition \ref{prop:Mg-is-indec}.
\end{proof}

\section{Isomorphism classes of sections}\label{sec:orbits}

In this section we exhibit examples showing that the number of isomorphism classes of the sections $\{\sigma_k: S^2 \to M: k \in \Z\}$ of $\pi: M \to S^2$ found in Theorem \ref{thm:infinite-sections} varies depending on the choices made in the construction.
Recall from the introduction that for two sections $s_1, s_2: S^2 \to M$ of $\pi: M \to S^2$, the pairs $(\pi, s_1)$ and $(\pi, s_2)$ are \emph{isomorphic} if there exist orientation-preserving diffeomorphisms $\Psi: M \to M$ and $\psi: S^2 \to S^2$ such that the following diagram commutes:
\begin{center}
\begin{tikzcd}
M \arrow[r, "\Psi"] \arrow[d, "\pi"]              & M \arrow[d, "\pi"']               \\
S^2 \arrow[r, "\psi"] \arrow[u, "s_1", bend left] & S^2 \arrow[u, "s_2"', bend right]
\end{tikzcd}
\end{center}
A \emph{positive factorization} is a factorization of some mapping class consisting only of right-handed Dehn twists $T_\ell$. According to Baykur--Hayano \cite[Theorem 1.1]{baykur-hayano}, there is a bijection between isomorphism classes of genus-$g$ Lefschetz fibrations with sections and positive factorizations of the identity in $\Mod(\Sigma_{g,1})$ up to Hurwitz equivalence. Two positive factorizations are said to be \emph{Hurwitz equivalent} if one can be obtained from the other by a sequence of two types of moves:
\begin{enumerate}[(a)]
\item \emph{(Elementary transformation)} For any $1 \leq i \leq r$,
\[
	T_{\ell_r} \dots T_{\ell_{i+1}}  T_{\ell_{i}} \dots T_{\ell_1} \longleftrightarrow T_{\ell_r} \dots (T_{\ell_{i+1}} T_{\ell_{i}} T_{\ell_{i+1}}^{-1}) T_{\ell_{i+1}} \dots T_{\ell_1}.
\]
\item \emph{(Global conjugation)} For any $f \in \Mod(\Sigma_{g,1})$,
\[
	T_{\ell_r} \dots T_{\ell_1} \longleftrightarrow (f T_{\ell_r} f^{-1}) \dots  (f T_{\ell_1}f^{-1}).
\]
\end{enumerate}
Above, note the equality of mapping classes $f T_{\ell_i} f^{-1} = T_{f(\ell_i)}$. Let $\sim$ denote Hurwitz equivalence of positive factorizations.

All examples of this section are untwisted fiber sums with infinitely many sections as constructed in Corollary \ref{cor:untwisted}.

\subsection{On the number of isomorphism classes in $\{\sigma_k: k \in \Z\}$}\label{sec:number-isom}

In this section we record examples of Lefschetz fibrations and infinitely many homologically distinct sections which are (1) pairwise isomorphic (Example \ref{ex:surj-monodromy}), and (2) pairwise non-isomorphic (Example \ref{ex:infty-orbits}).

A lemma of Auroux shows that in fact, the monodromy factorizations of the sections constructed in Corollary \ref{cor:untwisted} are often Hurwitz equivalent.
\begin{lem}[{Auroux, \cite[Lemma 6(b)]{auroux}}]\label{lem:surj-monodromy}
Let $g \geq 2$ and consider a positive factorization
\[
	T_{\ell_r} \dots T_{\ell_1} = 1 \in \Mod(\Sigma_{g,1}).
\]
Suppose that $\langle T_{\ell_1}, \dots, T_{\ell_r} \rangle = \Mod(\Sigma_{g,1})$. For any $\gamma \in \pi_1(\Sigma_g, p)$ and any $k \in \Z$, there is a Hurwitz equivalence
\[
	(T_{\ell_r} \dots T_{\ell_1})(T_{\ell_r} \dots T_{\ell_1})\sim (T_{\ell_r} \dots T_{\ell_1})(T_{P_\gamma(\ell_r)} \dots T_{P_\gamma^k(\ell_1)}) .
\]
\end{lem}
\begin{proof}
By assumption, $P_{\gamma} \in \Mod(\Sigma_{g,1})$ is contained in $\langle T_{\ell_1}, \, \dots, \, T_{\ell_r} \rangle$. By \cite[Lemma 6(b)]{auroux}, the factorization $(T_{\ell_r} \dots T_{\ell_1})$ is Hurwitz equivalent to $(T_{P_\gamma^k(\ell_r)} \dots T_{P_{\gamma}^k(\ell_1)})$ through a sequence of elementary transformations.
\end{proof}

\begin{example}\label{ex:surj-monodromy}
For a simple example of a positive factorization of the identity in $\Mod(\Sigma_{g,1})$ whose factors generate $\Mod(\Sigma_{g,1})$, consider the generators $T_{c_i}$ of $\Mod(\Sigma_{g,1})$, where $c_0, \dots, c_{2g+1} \subseteq \Sigma_{g,1}$ are depicted in Figure \ref{fig:curves-isomorphisms} (cf. \cite[Section 4.4.4]{farb-margalit}). Applying the $(2g+1)$-chain relation \cite[Proposition 4.12]{farb-margalit} twice shows that
\[
	(T_{c_1} \dots T_{c_{2g+1}})^{2g+2} (T_{f(c_1)} \dots T_{f(c_{2g+1})})^{2g+2} = 1 \in \Mod(\Sigma_{g,1}),
\]
where $f \in \Mod(\Sigma_{g, 1})$ is such that $f(c_1) = c_0$. Other examples of such factorizations can also be found in \cite[Section 3]{auroux}.

More generally, consider any genus-$g$ Lefschetz fibration $\pi: M \to S^2$ and section $s: S^2 \to M$ whose monodromy representation $\rho_{(\pi, s)}$ is surjective onto $\Mod(\Sigma_{g,1})$. Consider a monodromy factorization
\[
	T_{\ell_r} \dots T_{\ell_1} = 1 \in \Mod(\Sigma_{g,1})
\]
of $(\pi, s)$ and consider the sections $\{\sigma_k: S^2 \to M \#_F M : k \in \Z\}$ of the untwisted fiber sum $\pi \#_F \pi: M \#_F M \to S^2$ obtained in Corollary \ref{cor:untwisted} with monodromy factorizations
\[
	\left(T_{\ell_r} \dots T_{\ell_1}\right)\left(T_{P_\gamma^{k}(\ell_r)} \dots T_{P_\gamma^{k}(\ell_1)} \right)  = 1 \in \Mod(\Sigma_{g,1})
\]
for some $\gamma \in \pi_1(\Sigma_g, p)$. Lemma \ref{lem:surj-monodromy} shows that the monodromy factorizations of $(\pi, \sigma_0)$ and $(\pi, \sigma_k)$ are Hurwitz equivalent for all $k \in \Z$, and so $(\pi, \sigma_0)$ and $(\pi, \sigma_k)$ are isomorphic for all $k \in \Z$ by Baykur--Hayano \cite[Theorem 1.1]{baykur-hayano}.
\end{example}

\begin{figure}
\centering
\includegraphics[width=0.7\textwidth]{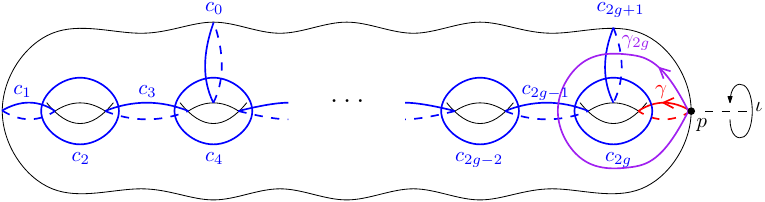}
\caption{The blue curves define the isotopy classes $c_0, \dots, c_{2g+1}$ of the curves in $\Sigma_{g,1}$ where $p$ is the marked point. The loops $\gamma_{2g}$ and $\gamma$ define elements of $\pi_1(\Sigma_g, p)$. The hyperelliptic involution $\iota$ fixes the marked point $p$ and preserves the isotopy classes $c_1, \dots, c_{2g}$.}\label{fig:curves-isomorphisms}
\end{figure}
\begin{rmk} \Cref{lem:surj-monodromy} suggests that the size of the monodromy group of a Lefschetz fibration relative to the point pushing subgroup $\pi_1(\Sigma_g, p) \leq \Mod(\Sigma_{g,1})$ controls the number of isomorphism classes of sections constructed in this paper. In light of this, we study hyperelliptic Lefschetz fibrations (cf. Brendle--Margalit \cite[Theorem 3.1]{brendle-margalit}) below to find examples of Lefschetz fibrations with infinitely many isomorphism classes of sections.
\end{rmk}

Before proceeding with the next example, we state an algebraic lemma whose proof is postponed to the end of this subsection.
\begin{lem}\label{lem:not-conjugate-monodromy}
For any $k \in \Z$, let $G_k \leq \Mod(\Sigma_{g,1})$ denote the subgroup
\[
	G_k := \langle T_{c_1}, \, \dots, \, T_{c_{2g}}, \, P_{\gamma^{-k}(\gamma_{2g}\gamma)^k} \rangle \leq \Mod(\Sigma_{g,1})
\]
where $c_1, \dots, c_{2g}, \gamma$, and $\gamma_{2g}$ are as shown in Figure \ref{fig:curves-isomorphisms}. If $\lvert k_1 \rvert \neq \lvert k_2 \rvert$ then $G_{k_1}$ and $G_{k_2}$ are not conjugate in $\Mod(\Sigma_{g,1})$.
\end{lem}
Using Lemma \ref{lem:not-conjugate-monodromy} the following example exhibits a Lefschetz fibration with infinitely many pairwise non-isomorphic sections.
\begin{example}\label{ex:infty-orbits}
For any $g \geq 2$, consider the genus-$g$ Lefschetz fibration $\pi: M \to S^2$ and section $s: S^2 \to M$ determined by the $(2g)$-chain relation \cite[Proposition 4.12]{farb-margalit}
\[
	(T_{c_1} \dots T_{c_{2g}})^{4g+2} = 1 \in \Mod(\Sigma_{g,1}).
\]
Consider the sections $\{\sigma_k: S^2 \to M \#_F M :  k \in \Z\} $ of the untwisted fiber sum $\pi \#_F \pi: M \#_F M \to S^2$ obtained in Corollary \ref{cor:untwisted} with monodromy factorization
\[
	 (T_{c_1} \dots T_{c_{2g}}) (T_{P_\gamma^{k}(c_1)}\dots T_{P_\gamma^{k}(c_{2g})}) = 1 \in \Mod(\Sigma_{g,1}).
\]
The image of the monodromy representation of $(\pi\#_F \pi, \sigma_k)$ is
\[
	G_k := \langle T_{c_1}, \dots, \, T_{c_{2g}}, \, T_{P_{\gamma}^k(c_{2g})}\rangle = \langle T_{c_1}, \dots, \, T_{c_{2g}}, \, P_{\gamma^{-k}(\gamma_{2g}\gamma)^k}\rangle \leq \Mod(\Sigma_{g,1}),
\]
where the last equality follows because $T_{c_{2g}}T_{P_{\gamma}^k(c_{2g})}^{-1} = P_{\gamma^{-k}(\gamma_{2g}\gamma)^k}$. By Lemma \ref{lem:not-conjugate-monodromy}, the subgroups $G_{k_1}$ and $G_{k_2}$ are not conjugate in $\Mod(\Sigma_{g,1})$ if $\lvert k_1 \rvert \neq \lvert k_2 \rvert$. Therefore the monodromy factorizations of $(\pi \#_F \pi, \sigma_{k_1})$ and $(\pi \#_F \pi, \sigma_{k_2})$ are not Hurwitz equivalent, so the sections $(\pi \#_F \pi, \sigma_{k_1})$ and $(\pi \#_F \pi, \sigma_{k_2})$ are not isomorphic if $\lvert k_1 \rvert \neq \lvert k_2 \rvert$ by Baykur--Hayano \cite[Theorem 1.1]{baykur-hayano}.
\end{example}

It remains to prove Lemma \ref{lem:not-conjugate-monodromy}. To do so, we determine the kernel of the forgetful map $\Mod(\Sigma_{g,1}) \to \Mod(\Sigma_g)$ in the Birman exact sequence (\ref{eqn:birman}) restricted to $G_k$ for each $k\in \Z$.
\begin{lem}\label{lem:ker-Gk}
Let $H_k \trianglelefteq G_k$ be the subgroup normally generated by $P_{\gamma^{-k}(\gamma_{2g}\gamma)^k}$ in $G_k$, i.e.
\[
	H_k := \langle f P_{\gamma^{-k}(\gamma_{2g}\gamma)^k} f^{-1} : f \in G_0 \rangle.
\]
The kernel of the map $\mathrm{Forget}: \Mod(\Sigma_{g,1}) \to \Mod(\Sigma_g)$ of the Birman exact sequence (\ref{eqn:birman}) restricted to $G_k \leq \Mod(\Sigma_{g,1})$ is $H_k$. In other words, the Birman exact sequence restricts to give a split short exact sequence
\[
	1 \to H_k \to G_k \to G_0 \to 1.
\]
\end{lem}
\begin{proof}
Denote $P_{\gamma^{-k}(\gamma_{2g}\gamma)^k}$ by $P$. Take any $f \in G_k$ and write
\[
	f = g_1  P^{m_1}  g_2  \dots  g_r  P^{m_r}  g_{r+1} \in G_k = \langle G_0, \, P \rangle
\]
for some $m_1, \dots, m_r \in \Z$ and $g_1, \dots, g_{r+1} \in G_0$. By repeatedly applying the identity
\[
	g P^{m}  g' = (g P^{m} g^{-1})(g g')
\]
for any $g, g' \in G_0$, we may move all factors $P^{m_i}$ to the left after possibly first conjugating them by elements of $G_0$. In other words, $G_k = H_k G_0$.

Let $\iota \in \Diff^+(\Sigma_{g,1})$ be the hyperelliptic involution shown in Figure \ref{fig:curves-isomorphisms}. Then $G_0$ is contained in the hyperelliptic mapping class group $\SMod(\Sigma_{g,1}) \leq \Mod(\Sigma_{g,1})$ determined by $\iota$, so the restriction of the forgetful map $\Mod(\Sigma_{g,1})\to \Mod(\Sigma_g)$ to the subgroup $G_0 \leq \Mod(\Sigma_{g,1})$ is injective by a result of Brendle--Margalit \cite[Theorem 3.1]{brendle-margalit}. Identifying $G_0$ with its image in $\Mod(\Sigma_g)$, the restriction of the Birman exact sequence (\ref{eqn:birman}) to $G_k \leq \Mod(\Sigma_{g,1})$ is
\[
	1 \to \pi_1(\Sigma_g,p) \cap G_k \to G_k \to G_0 \to 1.
\]
There is an inclusion $H_k \subseteq \pi_1(\Sigma_g, p)$ by construction. Conversely, take any $f \in \pi_1(\Sigma_g,p) \cap G_k$ and write $f = h g_0$ for some $h \in H_k$, $g_0 \in G_0$. Then $f$ is an element of $H_k$ because $g_0 = h^{-1} f$ is contained in $\pi_1(\Sigma_g, p)\cap G_0 = 1$
\end{proof}

\begin{proof}[Proof of Lemma \ref{lem:not-conjugate-monodromy}]
Identify $\pi_1(\Sigma_g,p)$ with its image under the homomorphism $\mathrm{Push}: \pi_1(\Sigma_g, p)\to \Mod(\Sigma_{g,1})$ in the Birman exact sequence (\ref{eqn:birman}) so that $P_{\gamma^{-k} (\gamma_{2g}\gamma)^k}$ is identified with $ (\gamma_{2g}\gamma)^{-k}\gamma^{k} \in \pi_1(\Sigma_g, p)$. Consider the abelianization map $h: \pi_1(\Sigma_g,p) \to H_1(\Sigma_g; \Z)$. The image $h(H_k)$ lies in $k H_1(\Sigma_g; \Z)$ because $h(f_*((\gamma_{2g}\gamma)^{-k}\gamma^{k}) = -k\left(h(f_*(\gamma_{2g}))\right)$ for any $f \in \Mod(\Sigma_{g,1})$. Moreover, $h(f_*(\gamma_{2g}))$ is a primitive element of $H_1(\Sigma_g; \Z)$ for any $f \in \Mod(\Sigma_{g,1})$ because $h(\gamma_{2g}) = \ell_{2g}$ is primitive. Therefore if $\lvert k_1 \rvert \neq \lvert k_2 \rvert$,
\[
	h(f H_{k_1} f^{-1}) = f_*(h(H_{k_1})) \neq h(H_{k_2}) \leq H_1(\Sigma_g; \Z).
\]
Finally, conclude by noting that $f H_{k_1} f^{-1} = \pi_1(\Sigma_g, p) \cap (f G_{k_1} f^{-1})$ by Lemma \ref{lem:ker-Gk}.
\end{proof}

\subsection{Non-isomorphisms via fiberwise diffeomorphisms covering the identity}

In this subsection we consider isomorphisms of pairs $(\pi, s_1)$ and $(\pi, s_2)$ via fiberwise diffeomorphisms covering the identity $\Id: S^2 \to S^2$, i.e. diffeomorphisms $\Psi \in \Diff^+(M)$ making the following diagram commute:
\begin{figure}[H]
\centering
\begin{tikzcd}
M \arrow[rd, "\pi"] \arrow[rr, "\Psi"] & & M \arrow[ld, "\pi"'] \\
& S^2 \arrow[ru, "s_2"', bend right] \arrow[lu, "s_1", bend left] &
\end{tikzcd}
\end{figure}
The following proposition shows that sections that are not isomorphic via such diffeomorphisms are common.
\begin{prop}\label{prop:covering-identity}
For any $g \geq 2$, let $\pi: M \to S^2$ be a genus-$g$ Lefschetz fibration with a section $s: S^2 \to M$. The untwisted fiber sum $\pi \#_F \pi: M \#_F M \to S^2$ yields infinitely many, homologically distinct sections
\[
	\{\sigma_k: S^2 \to M \#_F M : k \in \Z_{\geq 0}\}
\]
so that if $k_1 \neq k_2 \in \Z_{\geq 0}$ then $(\pi \#_F \pi, \sigma_{k_1})$ and $(\pi \#_F \pi, \sigma_{k_2})$ are not isomorphic via a fiberwise diffeomorphism $\Psi \in \Diff^+(M \#_F M)$ covering the identity of $S^2$.
\end{prop}
\begin{proof}
Let $T_{\ell_r} \dots T_{\ell_{1}} = 1 \in \Mod(\Sigma_{g,1})$ be a monodromy factorzation of the pair $(\pi, s)$. By work of Stipsicz \cite[Theorem 1.3]{stipsicz-vanishing-cycles}, $\pi: M \to S^2$ has a nonseparating vanishing cycle. By the change-of-coordinates principle \cite[Section 1.3]{farb-margalit} and after possibly applying some Hurwitz moves, we may assume that $\ell_1 \subseteq \Sigma_{g,1}$ is the curve shown in Figure \ref{fig:geom-intersection}. Pick $\gamma \in \pi_1(\Sigma_g, p)$ to intersect $\ell_1$ exactly once as shown in Figure \ref{fig:geom-intersection}.
\begin{figure}
\centering
\includegraphics[width=0.3\textwidth]{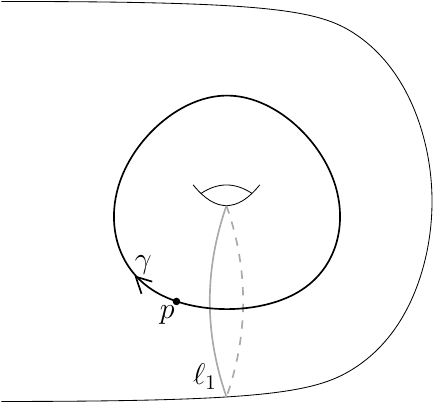}\qquad
\includegraphics[width=0.3\textwidth]{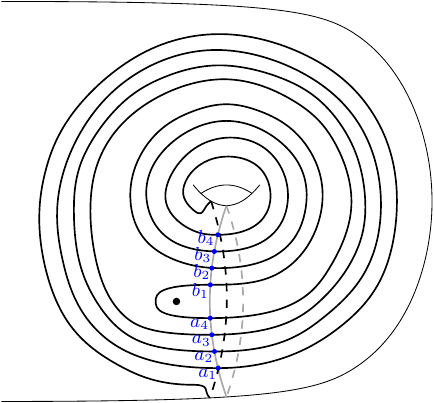}\qquad
\includegraphics[width=0.3\textwidth]{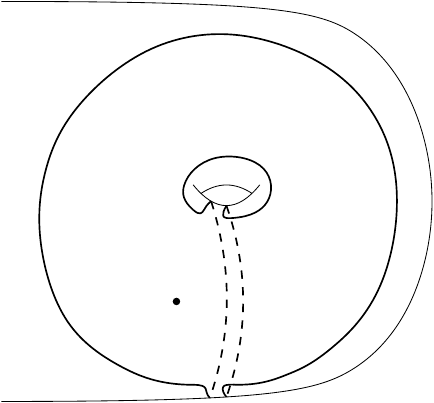}
\caption{Left: The curve $\ell_1 \subseteq \Sigma_{g,1}$ and loop $\gamma \in \pi_1(\Sigma_g, p)$ intersect once transversely as shown. Middle: The case of $k = 4$. The curves $P_\gamma^4(\ell_1)$ (in black) and $\ell_1$ (in grey) intersect in the points $a_1, \dots, a_k, b_1, \dots, b_k$. Right: The shown curve is isotopic to the curve $\alpha' \cup \beta'$ formed in case (c) of the proof of Proposition \ref{prop:covering-identity}.}\label{fig:geom-intersection}
\end{figure}

For any $k \geq 0$, consider the curves $\alpha$ and $\beta$ representing the isotopy classes $P_\gamma^k(\ell_1)$ and $\ell_1$ respectively and the intersection points $a_1, \dots, a_k$, $b_1, \dots, b_k \in \Sigma_{g,1}$ as shown in Figure \ref{fig:geom-intersection}. Consider pairs of intersection points that are adjacent in both $\alpha$ and $\beta$, and let $\alpha' \subseteq \alpha$ and $\beta' \subseteq \beta$ be the subarcs connecting them.
\begin{enumerate}[(a)]
\item If $\partial \alpha' = \partial \beta' = \{a_k, b_1\}$ then $\alpha' \cup \beta'$ bounds a punctured disk $S$. The complement $\Sigma_{g,1}- S$ is homeomorphic to the punctured surface $\Sigma_{g,1}$.
\item If $\partial \alpha' =\partial \beta' = \{a_i, a_{i+1}\}$ or $\{b_i, b_{i+1}\} $ then $\alpha' \cup \beta'$ is a nonseparating curve in $\Sigma_{g,1}$.
\item If $\partial \alpha' =\partial \beta' = \{a_1, b_k\}$ then $\alpha' \cup \beta'$ is a separating curve in $\Sigma_{g,1}$ shown in Figure \ref{fig:geom-intersection}. The subsurface $S \subseteq \Sigma_{g,1}$ bounded by $\alpha' \cup \beta'$ containing the marked point $p$ deformation retracts onto $\gamma \cup \ell_1$, and so $S$ is homeomorphic to $\Sigma_{1,1}^1$. The complement $\Sigma_{g,1} - S$ is then homeomorphic to $\Sigma_{g-1, 1}$.
\end{enumerate}
The casework above shows that $\alpha$ and $\beta$ do not form any bigons in $\Sigma_{g,1}$. By the bigon criterion \cite[Proposition 1.7]{farb-margalit}, the curves $\alpha$ and $\beta$ are in minimal position and the geometric intersection number $i(P_\gamma^k(\ell_1), \ell_1)$ is
\[
	i(P_\gamma^k(\ell_1), \ell_1) = \#(\alpha \cap \beta) = 2k.
\]

On the other hand, let $\sigma_k: S^2 \to M \#_F M$ be the section of $\pi \#_F \pi: M \#_F M \to S^2$ constructed in Corollary \ref{cor:untwisted} with monodromy factorization
\[
	\left(T_{\ell_r} \dots T_{\ell_1}\right) \left(T_{P_\gamma^k(\ell_r)} \dots T_{P_\gamma^k(\ell_1)} \right)  = 1 \in \Mod(\Sigma_{g,1}).
\]

If $\Psi \circ \sigma_{k_1} = \sigma_{k_2}$ for some fiberwise diffeomorphism $\Psi \in \Diff^+(M \#_F M)$ and $k_1, k_2 \in \Z_{\geq 0}$ then the monodromy factorizations of $(\pi \#_F\pi, \sigma_{k_1})$ and $(\pi\#_F \pi,\sigma_{k_2})$ are conjugate in $\Mod(\Sigma_{g,1})$, i.e. there exists $f \in \Mod(\Sigma_{g,1})$ so that the following equality of positive factorizations holds:
\[
	\left(T_{\ell_r} \dots T_{\ell_1}\right) \left(T_{P_\gamma^{k_1}(\ell_r)} \dots T_{P_\gamma^{k_1}(\ell_1)} \right) = \left((fT_{\ell_r}f^{-1}) \dots (fT_{\ell_1}f^{-1})\right) \left((f T_{P_\gamma^{k_2}(\ell_r)} f^{-1}) \dots (fT_{P_\gamma^{k_2}(\ell_1)}f^{-1}) \right) .
\]
In particular, there are equalities of isotopy classes of curves in $\Sigma_{g,1}$
\[
	P_\gamma^{k_1}(\ell_1) = f(P_\gamma^{k_2}(\ell_1)), \qquad \ell_1 = f(\ell_1).
\]
Then $k_1 = k_2$ because
\[
	2 k_1 = i(P_\gamma^{k_1}(\ell_1), \ell_1) = i(P_\gamma^{k_2}(\ell_1), \ell_1) = 2k_2. \qedhere
\]
\end{proof}

\bibliographystyle{alpha}
\bibliography{sections}

\end{document}